\documentclass[11pt]{article}

\usepackage{amsmath,amsthm,amssymb,enumitem,bbm,xcolor}
\usepackage[margin=1in]{geometry}
\usepackage{tikz}
\usepackage{etoolbox}
\usepackage{float}
\usepackage{todonotes}

\theoremstyle{plain}
\newtheorem{theorem}{Theorem}[section]

\newtheorem{corollary}[theorem]{Corollary}

\newtheorem{lemma}[theorem]{Lemma}

\newtheorem{proposition}[theorem]{Proposition}

\theoremstyle{definition}

\newtheorem*{claim*}{Claim}

\newcommand{\R}{\mathbb{R}}
\newcommand{\N}{\mathbb{N}}
\newcommand{\Z}{\mathbb{Z}}

\newcommand{\Complement}[1]{\overline{#1}}
\newcommand{\Dual}[1]{{#1}^*}
\newcommand{\Opt}{\textsc{Opt}}
\newcommand{\Fractional}[1]{{ #1 }_\mathrm{f}}
\usepackage{bbold}
\usepackage{cite}
\newbool{colours}
\booltrue{colours}
\ifbool{colours}
{ % If colours are on
\newcommand{\LPDual}[1]{ \textcolor{violet} { \textcolor{black}{#1}^* } }
\newcommand{\HDual} [1]{ \textcolor{mygreen}  { \textcolor{black}{#1}^* } }
\newcommand{\LPComplement}  [1]{ \textcolor{red}  { \overline{ \textcolor{black}{#1} }  } }
\newcommand{\HComplement}   [1]{ \textcolor{cyan}    { \overline{ \textcolor{black}{#1} }  } }
\newcommand{\DComplement}   [1]{ \textcolor{orange}    { \overline{ \textcolor{black}{#1} }  } }
}
{ % If colours are off
\newcommand{\LPDual}[1]{ {#1}^* }
\newcommand{\HDual} [1]{ {#1}^* }
\newcommand{\LPComplement}  [1]{ \overline{#1} }
\newcommand{\HComplement}   [1]{ \overline{#1} }
\newcommand{\DComplement}   [1]{ \overline{#1} }
}

\newcommand{\Define}[1]{\emph{#1}}

\newcommand{\InOpen}    [1] {{ #1 }^{\mathrm{in}}_{\mathrm{o}}}
\newcommand{\InClosed}  [1] {{ #1 }^{\mathrm{in}}_{\mathrm{c}}}
\newcommand{\OutOpen}   [1] {{ #1 }^{\mathrm{out}}_{\mathrm{o}}}
\newcommand{\OutClosed} [1] {{ #1 }^{\mathrm{out}}_{\mathrm{c}}}

\newcommand{\gammaIn}   {\gamma^{\mathrm{in}}}
\newcommand{\gammaOut}  {\gamma^{\mathrm{out}}}
\newcommand{\GammaIn}   {\Gamma^{\mathrm{in}}}
\newcommand{\GammaOut}  {\Gamma^{\mathrm{out}}}

\newcommand{\IP}[1]{{ #1 }^\mathbb{Z}}

\def\arraystretch{1.5}
\colorlet{mygreen}{green!50!black}

\title{Linear Programming Complementation}
\author{Maximilien Gadouleau\thanks{Department of Computer Science, Durham University, UK. Email: \texttt{m.r.gadouleau@durham.ac.uk}} 
\and George B. Mertzios\thanks{Department of Computer Science, Durham University, UK. Email: \texttt{george.mertzios@durham.ac.uk}} \thanks{Supported by the EPSRC grant EP/P020372/1.}
\and Viktor Zamaraev\thanks{Department of Computer Science, University of Liverpool, UK. Email: \texttt{viktor.zamaraev@liverpool.ac.uk}} \thanks{Supported by the EPSRC grant EP/P020372/1, during the time the author was at Durham University.}}
\date{\vspace{-0,3cm}}

\begin{document}

\maketitle

\begin{abstract}
In this paper we introduce a new operation for Linear Programming (LP), called \emph{LP complementation}, which resembles many properties of LP duality. 
Given a maximisation (resp.~minimisation) LP $P$, we define its \emph{complement} $Q$ as a specific minimisation (resp.~maximisation) LP which has the \emph{same} objective function as $P$. Our central result is the LP complementation theorem, that establishes the following relationship between the optimal value $\Opt(P)$ of $P$ and the optimal value $\Opt(Q)$ of its complement: $\frac{1}{\Opt(P)}+\frac{1}{\Opt(Q)}=1$. 
The LP complementation operation can be applied if and only if $P$ has an optimum value greater than 1. 
%Furthermore, similarly to LP duality, this operation is an involution.
% , i.e.~the complement of the complement of $P$ is again $P$. 

As it turns out, LP complementation has far-reaching implications. 
To illustrate this, we first apply LP complementation to \emph{hypergraphs}. 
For any hypergraph $H$, 
we review the four classical LPs, namely  
\emph{covering} $K(H)$, \emph{packing} $P(H)$, \emph{matching} $M(H)$, and \emph{transversal} $T(H)$. 
For every hypergraph $H=(V,E)$, we call $\HComplement{H}=(V,\{V\setminus e : e\in E\})$ the \emph{complement} of $H$. 
For each of the above four LPs, we relate the optimal values of the LP for the dual hypergraph $\HDual{H}$ to that of the complement hypergraph $\HComplement{H}$ (e.g.~$\frac{1}{ \Opt( K(\HDual{H}) ) }+\frac{1}{\Opt( K(\HComplement{H}) ) } = 1$).

As a second application of LP complementation, we consider \emph{fractional graph theory}.
We prove that the LP for the \Define{fractional in-dominating number} of a digraph $D$ is the complement of the LP for the \Define{fractional total out-dominating number} of the digraph complement $\DComplement{D}$ of $D$. 
Furthermore we apply the hypergraph complementation theorem to matroids. 
We establish that the fractional matching number of a matroid coincide with its edge toughness.
This result can then be applied to graphic matroids, yielding a formula for the edge toughness of a graph.
Moreover, we derive an alternative proof of the relationship between the edge toughness of a matroid and the fractional covering number of its dual matroid.

As our last application of LP complementation, we introduce the natural problem \textsc{Vertex Cover with Budget (VCB)}: for a graph $G=(V,E)$ and a positive integer $b$, what is the maximum number $t_b$ of vertex covers 
$S_1, \dots, S_{t_b}$ of $G$, such that every vertex $v\in V$ appears in at most $b$ vertex covers? 
The integer $b$ can be viewed as a ``budget'' that we can spend on each vertex and, given this budget, we aim to cover all edges for as long as possible. 
Note that $t_b\geq b$, as a trivial feasible solution to \textsc{VCB} is to select \emph{all} vertices in $V$ for $b$ consecutive times. That is, 
\textsc{VCB} is actually not a generalisation of \textsc{Minimum Vertex Cover}. 
We relate \textsc{VCB} with the LP $Q_G$ for the fractional chromatic number $\Fractional{\chi}$ of a graph $G$. 
More specifically, we prove that, as $b \to \infty$, the optimum for \textsc{VCB} satisfies $t_b \sim \Fractional{t} \cdot b$, where $\Fractional{t}$ is the optimal solution to the complement LP of $Q_G$.
Finally, our results imply that, for any finite budget $b$, it is NP-hard to decide whether $t_b\geq b + c$ for any $1 \le c \le b-1$. 
%Since, in every graph, a vertex cover is the complement of an independent set, we apply LP complementation and our results on hypergraphs to solve \textsc{VCB}. 

While our results about fractional domination and matroids are just special applications of LP complementation, they are quite general on their own as their special cases yield existing results from the literature.
This demonstrates the generality and power of LP complementation which we believe is yet to be fully understood.
\end{abstract}

\newpage 

\section{Introduction}

Many optimisation problems can be expressed as, or reduced to, Linear Programs (LPs) or Integer Programs (IPs) \cite{PS82}. As such, the use of Linear Programming is ubiquitous \cite{Sch03}, with applications in combinatorial optimisation, combinatorics, industrial engineering, coding theory, etc. One of the key aspects of Linear Programming is LP duality, and in particular the strong LP duality theorem which states that the optimal values of an LP is equal to that of its dual \cite{Sch86}.

Many classical problems from graph theory, e.g. maxmimum matching, minimum vertex cover, chromatic number, independence number, clique number, minimum dominating set, domatic number, etc. can be expressed as Integer Programs (IPs). Fractional graph theory then investigates these problems with three main approaches (see the book by Scheinerman and Ullman \cite{SU97} for a survey). First, it studies the Linear Programming (LP) relaxations of these IPs, some of which have found applications of their own, e.g fractional chromatic number or fractional domatic number for scheduling \cite{GH21, AELTW16}. Second, it applies LP techniques to either the original IP problems or their LP relaxations. Amongst those, LP duality is one of the most powerful and ubiquitous \cite{Sch86}. Third, it generalises the results to hypergraphs in order to get a clearer framework. In particular, hypergraph duality, where the roles of vertices and edges are swapped, is common practice.

\paragraph{Motivating example} We illustrate the main contributions of this paper, namely LP and hypergraph complementations, via a simple example first. Consider the following problem \textsc{Vertex Cover with Budget (VCB)}. Given a graph $G$ and a vertex budget $b \ge 1$, find the largest collection $S_1, \dots, S_{t_b}$ of vertex covers such that every vertex belongs to at most $b$ of the $S_i$'s. As $b$ tends to infinity, the optimum satisfies $t_b \sim \Fractional{t} \cdot b$, where $\Fractional{t}$ is defined as follows. Let $A$ be the incidence matrix of vertex covers of $G$, where $A_{ve} = 1$ if and only if $v$ belongs to the vertex cover $e$. Then $\Fractional{t} = \max\{1^\top x: Ax \le 1, x \ge 0 \}$. For instance, the pentagon $C_5$ has eleven vertex covers: all complements of non-edges and all sets of four or five vertices. We obtain
\[
A(C_5) = 
\begin{pmatrix}
    1 & 0 & 1 & 0 & 1 & 1 & 0 & 1 & 1 & 1 & 1 \\
    1 & 1 & 0 & 1 & 0 & 1 & 1 & 0 & 1 & 1 & 1 \\
    0 & 1 & 1 & 0 & 1 & 1 & 1 & 1 & 0 & 1 & 1 \\
    1 & 0 & 1 & 1 & 0 & 1 & 1 & 1 & 1 & 0 & 1 \\
    0 & 1 & 0 & 1 & 1 & 0 & 1 & 1 & 1 & 1 & 1 
\end{pmatrix},
\qquad 
\Fractional{t}(C_5) = \max\{1^\top x: Ax \le 1, x \ge 0 \} = \frac{5}{3}.
\]
% We see that the optimal time is given by $t_b = \max\{1^\top x: Ax \le b, x \in \Z^n \}$. 
% For instance, $\Fractional{t} = 5/3$ for $G = C_5$ (assign weight $1/3$ to each minimal vertex cover).
% Direct computation shows that for $C_5$, we obtain $\Fractional{t}(C_5) = 5/3$. 
In general, the $\Fractional{t}$ quantity has received very little attention. LP complementation then links it to the much more well studied fractional chromatic number of graphs \cite{SU97}.
A $c$-multicolouring of $G$ is the smallest size of a collection of independent sets $\Complement{S}_1, \dots, \Complement{S}_{\chi_c}$, such that each vertex belongs to at least $c$ of the $\Complement{S}_i$'s. As $c$ tends to infinity, the optimum satisfies $\chi_c \sim \Fractional{\chi} \cdot c$, where $\Fractional{\chi}$ is the fractional chromatic number of $G$. We have $\Fractional{\chi} = \min\{1^\top x: (1- A)x \ge 1, x \ge 0 \}$, where $1-A$ is the incidence matrix of independent sets of $G$. For instance, for $C_5$,
\[
1 - A(C_5) = 
\begin{pmatrix}
    0 & 1 & 0 & 1 & 0 & 0 & 1 & 0 & 0 & 0 & 0 \\
    0 & 0 & 1 & 0 & 1 & 0 & 0 & 1 & 0 & 0 & 0 \\
    1 & 0 & 0 & 1 & 0 & 0 & 0 & 0 & 1 & 0 & 0 \\
    0 & 1 & 0 & 0 & 1 & 0 & 0 & 0 & 0 & 1 & 0 \\
    1 & 0 & 1 & 0 & 0 & 1 & 0 & 0 & 0 & 0 & 0 
\end{pmatrix},  
\qquad
\Fractional{\chi}(C_5) = \min\{1^\top x: (1-A)x \ge 1, x \ge 0 \} = \frac{5}{2}.
\]
% and $\Fractional{\chi}(C_5) = 5/2$ (assign weight $1/2$ to each non-edge). 

The relation between vertex covers and independent sets is an example of hypergraph complementation, and accordingly, the $\Fractional{t}$ and $\Fractional{\chi}$ terms are examples of LP complementation. We shall prove that these two values form a complement pair, i.e. 
%$\Fractional{t}(G) = \frac{ \Fractional{\chi}(G) }{ \Fractional{\chi}(G) - 1 }$.
$\frac{1}{ \Fractional{t} } + \frac{1}{ \Fractional{\chi} } = 1$, which is easily verified for $C_5$. Therefore, computing the fractional chromatic number immediately yields the asymptotic behaviour of the vertex cover with budget problem.

\subsection*{Our contributions}

\paragraph{Linear Programming Complementation.}
In this paper we introduce the notion of the \Define{complement} of an LP $R$, which we denote by $\LPComplement{R}$, as follows. Let $c \in \R^n$, $b \in \R^m$, $A \in \R^{m \times n}$, then for the following maximisation LP $P$, we have
\begin{align*}
    P &:~ \max \{ c^\top x : Ax \le b \},\\
    \LPComplement{P} &:~ \min \{ c^\top x : (b c^\top - A) x \ge b \}.
\end{align*}
Similarly, let $v \in \R^n$, $u \in \R^m$, $M \in \R^{m \times n}$, then for the following minimisation LP $Q$, we have
\begin{align*}
    Q &:~ \min \{ v^\top x : Mx \ge u \},\\
    \LPComplement{Q} &:~ \max \{ v^\top x : (u v^\top - M) x \le u \}.
\end{align*}

To simplify notation, in the remainder of the paper we use the notation $P$ (resp.~$Q$) to denote a maximisation (resp.~minimisation) LP, while we use $R$ to denote an arbitrary LP which can be either a maximisation or a minimisation LP. 
Furthermore, for any linear program $R$, adding the constraint that the variables have to be integral yields an integer program, which we denote $\IP{R}$.

\bigskip
\noindent
\textbf{LP complementation theorem}. Our central result is a surprising relation between the optimal values of an LP and its complement, given that one of these values is finite and larger than 1.

\begin{theorem}[LP complementation theorem] \label{th:LP_complementation}
For any LP $R$, $1 < \Opt(R) < \infty$ if and only if $1 < \Opt(\LPComplement{R}) < \infty$, in which case 
\[
	\frac{1}{\Opt(R)} + \frac{1}{ \Opt( \LPComplement{R} ) } = 1.
\]
\end{theorem}

\noindent
Alternatively, the theorem states that the harmonic mean of the optimal values of the LP and its complement is 2. Consequently, the two values are separated by 2, and one value is equal to 2 if and only if the other is equal to 2.

% \todo[inline]{Get back to the example with the optimal values that are complement pairs.}

\bigskip
\noindent    
\textbf{Natural interpretation of LP complementation}.
The links between two-player zero-sum (matrix) games and LP are well established; see \cite{Vor77, Bri89} for instance. We shall review these and then show that LP complementation can be interpreted using two complementary games. 

Given any $m \times n$ matrix $A$, the matrix game $\Gamma_A$ with payoff matrix $A$ is played by two persons, Rose and Colin, as follows. Rose selects a row of $A$, Colin a column. If the row $i$ and the column $j$ are chosen, then Rose's payoff is $a_{ij}$. In particular, if $a_{ij} > 0$, then Rose earns money; otherwise, Rose loses money.  

% A strategy for Rose is then a probability distribution on the rows: $r = (r_1, \dots, r_m)$ such that $r \ge 0$ and $1^\top r = 1$. The expected payoff for strategy $r$ is then $v_r = \min_j \{ r^\top A_{\cdot j} \}$, where $A_{\cdot j}$ denotes the $j$-th column of $A$. Rose aims at maximising her expected payoff. The value of the game, denoted as $V$, is the maximum expected payoff over all strategies for Rose (and is equal to the minimum expected payoff over all strategies for Colin). 

A strategy for Colin is then a probability distribution on the columns: $c = (c_1, \dots, c_n)^\top$ such that $c \ge 0$ and $1^\top c = 1$. Rose's expected payoff for a given strategy $c$ for Colin is then $v_c = \max_i \{ A_i c \}$, where $A_i$ is the $i$-th row of $A_i$; thus $A c \le v_c \cdot 1$. Colin aims at minimising Rose's expected payoff. The value of the game, denoted as $V$, is the minimum expected of Rose's payoff over all strategies for Colin.

Without loss of generality, suppose that $0 \le A \le 1$. Then the value $V$ of the game is also between $0$ and $1$; let us omit the two extreme cases and suppose that $0 < V < 1$. For any strategy $c$ for Colin with payoff $v_c$, let $x = \frac{1}{ v_c } c$, then we have $x \ge 0$, $Ax \le 1$, and $1^\top x = \frac{1}{ v_c }$. Minimising Rose's expected payoff $v_c$ then corresponds to maximising $1^\top x$. We can then express $V =1/\Opt(P)$, where
\[
    P :~ \max \{ 1^\top x : Ax \le 1, x \ge 0 \}.
\]
LP duality then corresponds to taking Rose's point of view: $V = 1/\Opt(\LPDual{P})$, with
\[
    \LPDual{P} :~ \min \{ 1^\top y : A^\top y \ge 1, y \ge 0 \}.
\]

% Without loss of generality, suppose that $0 \le A \le 1$. Then the value $V$ of the game is also between $0$ and $1$; let us omit the two extreme cases and suppose that $0 < V < 1$. For any strategy $r$ for Rose with payoff $v_r$, let $x = \frac{1}{ v_r } r$, then we have $x \ge 0$, $Ax \le 1$, and $1^\top x = \frac{1}{ v_r }$. We can then express $V =1/\Opt(P)$, where
% \[
%     P :~ \max \{ 1^\top x : Ax \le 1, x \ge 0 \}.
% \]
% LP duality then corresponds to taking Colin's point of view: $V = 1/\Opt(\LPDual{P})$, with
% \[
%     \LPDual{P} :~ \min \{ 1^\top y : A^\top y \ge 1, y \ge 0 \}.
% \]

LP complementation, on the other hand, corresponds to taking the complementary payoff. Consider a second game, where the players change their roles (Rose chooses columns of the payoff matrix and Colin chooses rows), and the payoff is equal to $1$ minus the original payoff. Thus, the new payoff matrix is $(1 - A^\top)$ and the value of the new game is $\Complement{V} = 1 - V$. But then, we have $\Complement{V} = 1/\Opt(Q)$, where
\[
    Q = \LPComplement{P} :~ \min \{ 1^\top y : (1 - A) y \ge 1, y \ge 0 \}.
\]
We then have $\Opt(P) > 1$ and $\Opt( \LPComplement{P} ) > 1$ and
\[
	\frac{1}{\Opt(P)} +  \frac{1}{\Opt( \LPComplement{P} )} = 1.
\]

\bigskip  
\noindent
\textbf{Consequence for integer programming \& Bounds}. Let $P$ be a maxmisation LP. LP duality can be naturally used to study $P$, as any feasible solution to the dual $\LPDual{P}$ gives an upper bound on the optimal value of $P$. However, a feasible solution to the dual does not provide much information about feasible solutions of the primal. LP complementation works differently, as a feasible solution to the complement immediately yields a feasible solution to the primal by simple scaling. However, it only gives a lower bound on the optimal value. The primal and its complement then ``work together'' towards their optimal solutions and values.

The relationship between feasible solutions to the primal and the complement has some important consequences for IPs. Firstly, from $P$ and $\Complement{P}$, we obtain four programs $P_s$, $\IP{P_s}$, $(\LPComplement{P})_t$, and $\IP{ ( \LPComplement{P} )_t }$ (where $s$ and $t$ come from an optimal solution of $P$ and its value), which have a common optimal solution--see Corollary \ref{corollary:IP}.

Secondly, we introduce the bounds $\alpha(\IP{P})$ and $\beta(\IP{\LPComplement{P}})$ on the optimal values of $P$ and $\LPComplement{P}$, respectively. These bounds are based on feasible solutions of $\IP{P}$ and $\IP{\LPComplement{P}}$, respectively. We then prove that these bounds are ``mutually tight'' for the primal-complement pair--they are actually tight for the vertex cover with budget problem on $C_5$.

\begin{theorem} \label{theorem:alpha1}
Let $P : \max\{c^\top x : Ax \le b\}$, where $b > 0$ and $A \ne 0$, such that $1 < \Opt(\IP{P}) \le \Opt( \IP{\LPComplement{P}} ) < \infty$. Then $\Opt( \IP{P} ) \le \alpha(\IP{P}) \le \Opt( P )$, $\Opt( \LPComplement{P} ) \le \beta( \IP{ \LPComplement{P} } ) \le \Opt( \IP{ \LPComplement{P} } )$ and
\[
    \Opt(P) = \alpha(\IP{P}) \iff \Opt(\LPComplement{P}) = \beta(\IP{ \LPComplement{P} }).
\]
\end{theorem}

\bigskip
\noindent    
\textbf{Hypergraph complementation.} 
%\todo[inline]{Review standard notations from fractional hypergraph theory (covering, packing, matching, transversal), details in Section~\ref{section:hypergraphs}.}
%We now consider four LPs related to hypergraphs. 
For a hypergraph $H = (V, E)$ with $n$ vertices and $m$ edges, its \Define{incidence matrix} is denoted by $M_H \in \R^{n \times m}$. The \Define{dual} of the hypergraph $H$ is $\HDual{H} = (E, \Dual{V})$, where $\Dual{V} = \{ \{ e\in E: e \ni v\} : v \in V \}$. We then have $M_{\HDual{H}} = (M_H)^\top$ and $\HDual{(\HDual{H})} \cong H$. 
% We note that $H$ has no empty edge if and only if $\HDual{H}$ has no isolated vertex, and vice versa. 

Now we define the \Define{complement} of $H$ as $\HComplement{H} = (V, \{V \setminus e : e \in E \})$; note that $M_{\HComplement{H}} = 1 - M_H$.
% , where ``1'' denotes here the all-ones-matrix.
Hypergraph complementation is an involution that commutes with duality, i.e.~$\HComplement{\HComplement{H}} = H$ and ${\HComplement{ \left( \HDual{H} \right) } = \HDual{ \left( \HComplement{H} \right) }}$. 

%We shall then apply the LP complementation theorem to them; all these LPs have an optimal value in $[1, \infty]$. Technically, if the optimal value is either $1$ or $\infty$, then the LP complementation theorem does not apply. However, we highlight these degenerate cases, which can easily be handled separately. By using the convention that $1$ and $\infty$ form a complement pair, we can then include these degenerate cases in our hypergraph complementation theorem.

A \Define{covering} of a hypergraph $H$ is a set of edges whose union is equal to its set of vertices $V$. The \Define{covering number} $k(H)$ of $H$ is the minimum size of a covering of $H$; this can be formulated as the optimum of an integer program. 
The \Define{fractional covering number} $\Fractional{k}(H)$ of $H$ is the optimal value of the LP $K(H)$  that is obtained by lifting the integrality constraints.
% We remark that $K(H)$ is feasible if and only if $H$ has no isolated vertices. As it turns out, if $K(H)$ is feasible, then it has an optimal solution. In that case, $\Fractional{k}(H) = \Opt(K(H)) \ge 1$, with strict inequality if and only if $H$ has no complete edges.

It can be easily shown that $\LPComplement{K( \HDual{H} )} = \LPDual{ K( \HComplement{H} ) }$. By applying the LP complementation theorem to $K(H)$, we obtain the hypergraph complementation theorem, as follows.

\begin{theorem}[Hypergraph complementation theorem] \label{theorem:hypergraph_complementation-small}
For any hypergraph $H$,
\[
	\frac{1}{\Fractional{k}(\HDual{H})} + \frac{1}{\Fractional{k}(\HComplement{H})} = 1.
\]
\end{theorem}

Applying LP duality and hypergraph duality to $K(H)$ yields four standard LPs $K(H), P(H), T(H), M(H)$ for hypergraphs, given in Table \ref{tab:standardLP} and related in Figure \ref{fig:8LPs-1} \cite{SU97}. By applying LP complementation to these four LPs, 
we obtain the four new LPs 
$\LPComplement{ K(H)}, \LPComplement{ P(H) }, \LPComplement{ T(H) }, \LPComplement{ M(H) }$. The new notions of LP complementation and hypergraph complementation allow us to establish a formal relation of these four LPs with the four original LPs; see Figure \ref{fig:8LPs-2} for an illustration. 

%\todo[inline]{Put a table instead of paragraphs for the four LPs in the book.}
\begin{table}[H]
\centering
\begin{tabular}{|p{7.5cm}|p{7.5cm}|}
\hline
 \textbf{Covering number}, $\Fractional{k}(H)$ 
 \newline\medskip
 \footnotesize{min \# edges to cover all vertices} 
 \newline 
 \normalsize{$K(H) :~\min\{ 1^\top x :  M_H x \ge 1, x \ge 0\}$} & 
 \textbf{Packing number}, $\Fractional{p}(H)$ 
 \newline\medskip
 \footnotesize{max \# vertices, no two in the same edge}
 \newline 
 \normalsize{$P(H) :~\max	\{	1^\top x : 	M_H^\top x 	\le 1, x \ge 0\}$}\\ \hline
 \textbf{Transversal number}, $\Fractional{\tau}(H)$
  \newline\medskip
  \footnotesize{min \# vertices to touch edges} 
  \newline
  \normalsize{$T(H) :~\min	\{	1^\top x : M_H^\top x \ge 1, x \ge 0 \}$}
  &
 \textbf{Matching number}, $\Fractional{\mu}(H)$ 
  \newline\medskip
  \footnotesize{max \# pairwise disjoint edges} 
  \newline
  \normalsize{$M(H) :~\max \{	1^\top x : 	M_H x \le 1, x \ge 0\}$}
  \\ \hline
\end{tabular}
\caption{Four standard LPs for hypergraphs: covering, packing, transversal, and matching numbers of a hypergraph.} \label{tab:standardLP}
\end{table}

% A \Define{packing} of $H$ is a set of vertices such that every edge contains at most one of those vertices. 
% The \Define{packing number} $p(H)$ of $H$ is the maximum size of a packing of $H$; this can be again formulated as the optimum of an integer program.
% The \Define{fractional packing number} $\Fractional{p}(H)$ of $H$ is the optimal value of the LP $P(H)$, which is the dual of the LP $K(H)$. 
% We remark that $P(H)$ is always feasible. However, $P(H)$ is bounded if and and only if $H$ has no isolated vertices. In that case, $\Fractional{p}(H) = \Opt(P(H)) > 1$ if and only if it has no complete edges. LP duality then yields $\Fractional{p}(H) = \Fractional{k}(H)$.%, whenever $H$ has no isolated vertices.

% A \Define{matching} of $H$ is a set of disjoint edges; it corresponds to a packing of $\HDual{H}$. The \Define{fractional matching number} is then $\Fractional{\mu}(H) = \Fractional{p}(\HDual{H})$, which is the optimal value of the LP $M(H) = P(\HDual{H})$.

% A \Define{transversal} of $H$ is a set of vertices such that every edge contains a vertex from that set; it corresponds to a covering of $\HDual{H}$. The \Define{fractional transversal number} is then the optimum value $\Fractional{\tau}(H) = \Fractional{k}(\HDual{H})$ of the LP $T(H) = K(\HDual{H})$. 
% Again, LP duality yields $\Fractional{\mu}(H) = \Fractional{\tau}(H)$. 
% In summary, for any $H$ we have $\Fractional{\tau}(H) = \Fractional{k}(\HDual{H}) = \Fractional{p}(\HDual{H}) = \Fractional{\mu}(H)$. 

\begin{figure}[H]
\begin{center}
\begin{tikzpicture}[scale = 1.5, every node/.style={scale=0.8}]
    \draw[very thick, color=violet] (0,0) -- (2,0);
    \draw[very thick, color=mygreen] (0,0) -- (0,2);
    \draw[very thick, color=mygreen] (2,0) -- (2,2);
    \draw[very thick, color=violet] (0,2) -- (2,2);

%    \draw[very thick, color=violet] (1,1) -- (3,1);
%    \draw[very thick, color=mygreen] (1,1) -- (1,3);
%    \draw[very thick, color=mygreen] (3,1) -- (3,3);
%    \draw[very thick, color=violet] (1,3) -- (3,3);
    
%    \draw[very thick, color=red] (0,0) -- (1,1);
%    \draw[very thick, color=red] (2,0) -- (3,1);
%   \draw[very thick, color=red] (0,2) -- (1,3);
%    \draw[very thick, color=red] (2,2) -- (3,3);

    \node[anchor=east]  at (0, 0) {$K(H)$};
%    \node[anchor=east] at (0,-0.2) {$\min\{ 1^\top x: M_Hx \ge 1 \}$};
    \node[anchor=west]  at (2, 0) {$P(H) = \LPDual{ K(H) }$};  
    \node[anchor=east]  at (0,2) {$T(H) = K( \HDual{ H } )$};  
    \node[anchor=west]  at (2, 2) {$M(H) = \LPDual{ K( \HDual{H} ) }$};  
%    \node[anchor=east]  at (1, 1) {$\LPComplement{ K(H) } = M(\HComplement{H})$};  
%    \node[anchor=west]  at (3,1) {$\LPComplement{ P(H) } =  T(\HComplement{H})$};  
%    \node[anchor=east]  at (1,3) {$\LPComplement{ T(H) } = P(\HComplement{H})$};  
%    \node[anchor=west]  at (3, 3) {$\LPComplement{ M(H) } = K(\HComplement{H})$};  
    
    \node[draw=none] at (0,0) {\textbullet};
    \node[draw=none] at (0,2) {\textbullet};
    \node[draw=none] at (2,0) {\textbullet};
    \node[draw=none] at (2,2) {\textbullet};
%    \node[draw=none] at (1,1) {\textbullet};
%    \node[draw=none] at (1,3) {\textbullet};
%    \node[draw=none] at (3,1) {\textbullet};
%    \node[draw=none] at (3,3) {\textbullet};

\end{tikzpicture}

\small{ \textcolor{violet}{LP duality} ~~ \textcolor{mygreen}{Hypergraph duality}}

\caption{The four initial Linear Programs related to a hypergraph.} \label{fig:8LPs-1}
\end{center}
\end{figure}

\begin{figure}[H]
\begin{center}
\begin{tikzpicture}[scale = 1.5, every node/.style={scale=0.8}]
    \draw[very thick, color=violet] (0,0) -- (2,0);
    \draw[very thick, color=mygreen] (0,0) -- (0,2);
    \draw[very thick, color=mygreen] (2,0) -- (2,2);
    \draw[very thick, color=violet] (0,2) -- (2,2);

    \draw[very thick, color=violet] (1,1) -- (3,1);
    \draw[very thick, color=mygreen] (1,1) -- (1,3);
    \draw[very thick, color=mygreen] (3,1) -- (3,3);
    \draw[very thick, color=violet] (1,3) -- (3,3);
    
    \draw[very thick, color=red] (0,0) -- (1,1);
    \draw[very thick, color=red] (2,0) -- (3,1);
   \draw[very thick, color=red] (0,2) -- (1,3);
    \draw[very thick, color=red] (2,2) -- (3,3);

    \node[anchor=east]  at (0, 0) {$K(H)$};  
    \node[anchor=west]  at (2, 0) {$P(H)$};  
    \node[anchor=east]  at (0,2) {$T(H)$};  
    \node[anchor=west]  at (2, 2) {$M(H)$};  
    \node[anchor=east]  at (1, 1) {$\LPComplement{ K(H) } = M(\HComplement{H})$};  
    \node[anchor=west]  at (3,1) {$\LPComplement{ P(H) } =  T(\HComplement{H})$};  
    \node[anchor=east]  at (1,3) {$\LPComplement{ T(H) } = P(\HComplement{H})$};  
    \node[anchor=west]  at (3, 3) {$\LPComplement{ M(H) } = K(\HComplement{H})$};  
    
    \node[draw=none] at (0,0) {\textbullet};
    \node[draw=none] at (0,2) {\textbullet};
    \node[draw=none] at (2,0) {\textbullet};
    \node[draw=none] at (2,2) {\textbullet};
    \node[draw=none] at (1,1) {\textbullet};
    \node[draw=none] at (1,3) {\textbullet};
    \node[draw=none] at (3,1) {\textbullet};
    \node[draw=none] at (3,3) {\textbullet};

\end{tikzpicture}

\small{ \textcolor{violet}{LP duality} ~~ \textcolor{red}{LP complementation} ~~ \textcolor{mygreen}{Hypergraph duality}   ~~ \textcolor{cyan}{Hypergraph complementation} }

\caption{The eight Linear Programs related to a hypergraph.} \label{fig:8LPs-2}
\end{center}
\end{figure}

% For any $S \subseteq V$, we define $\rho_H(S)	= \max \left\{ |S \cap e| : e \in E \right\}$ and $\alpha(H) = \max \left\{ \frac{ |S| }{ \rho_H(S) } : S \subseteq V, \rho_H(S) > 0 \right\}$.
% Similarly, for any $Z \subseteq E$, we define 
% 	$\sigma_H(Z) 	= \min \left\{ |\{ e \in Z: v \in e \}| : v \in V \right\}$, and
% 	$\beta(H) 		= \min \left\{  \frac{ |Z| }{ \sigma_H(Z)} : Z \subseteq E, \sigma_H(Z) > 0 \right\}$.
% We immediately recognise that $\alpha(H) = \alpha( \IP{P(H)} )$ and $\beta(H) = \beta( \IP{K(H)} )$. 

The hypergraph complementation theorem then holds for all four parameters in Table \ref{tab:standardLP}.

\begin{corollary}
For any hypergraph $H$, we have
\[
	\frac{1}{ \Fractional{k}(\HDual{H}) } + \frac{1}{ \Fractional{k}(\HComplement{H}) } = 
	\frac{1}{ \Fractional{p}(\HDual{H}) } + \frac{1}{ \Fractional{p}(\HComplement{H}) } = 
	\frac{1}{ \Fractional{\mu}(\HDual{H}) } + \frac{1}{ \Fractional{\mu}(\HComplement{H}) } = 
	\frac{1}{ \Fractional{\tau}(\HDual{H}) } + \frac{1}{ \Fractional{\tau}(\HComplement{H}) } = 1.
\]
\end{corollary}

%\todo[inline]{Sentence about the $\alpha$ and $\beta$ bounds for hypergraphs.}
 
% Applying Theorem \ref{theorem:hypergraph_complementation} to $\HDual{H}$, we obtain the following corollary.

% \begin{corollary}
% For any hypergraph $H$, we have
% \[
% 	\mu(H) \le \alpha( \HDual{H} ) \le \Fractional{\mu}(H) = \Fractional{\tau}(H) \le \beta( \HDual{H} ) \le \tau(H).
% \]
% Moreover, $\Fractional{\mu}(\HComplement{H}) = \alpha(\HComplement{H})$ if and only if $\Fractional{\mu}(\HDual{H}) = \beta(\HDual{H})$.
% \end{corollary}

\bigskip
\noindent    
\textbf{The impact of LP complementation to related problems.} 
Here we give a brief overview of the implications that Linear Complementation has in the following two case studies. Full details are given in Sections~\ref{section:applications} and~\ref{section:budget}, respectively.

\paragraph{Case study 1: Fractional graph theory.}

We give two applications of the hypergraph complementation theorem to graph theory.

Firstly, fractional domination in digraphs provides a setting where LP complementation and hypergraph complementation naturally arise. We relate the \Define{fractional in-dominating number} of a digraph $D$ and the \Define{fractional total out-dominating number} of its digraph complement $\DComplement{D}$ as follows. 

\begin{theorem}[Domination complementation theorem] \label{th:domination}
For any digraph $D$, we have that $\frac{ 1 } { \Fractional{\gammaIn}(D) } + \frac{ 1 } { \Fractional{\GammaOut}(\DComplement{D}) } = 1$.
\end{theorem}

This theorem is very general, as it holds for all digraphs, and provides more specific relations about domination numbers for graphs, tournaments, and regular digraphs. The last one is itself a generalisation of the result in \cite[Theorem 7.4.1]{SU97}, which only applies to regular \emph{graphs}.

% The domination complementation theorem is a vast generalisation of the result in \cite[Theorem 7.4.1]{SU97}, which only applies to regular graphs, to all digraphs. 

Secondly, we apply the hypergraph complementation theorem to matroids. 
We establish that the fractional matching number of a matroid coincide with its edge toughness.
This result can then be applied to graphic matroids, yielding a formula for the edge toughness of a graph.
Moreover, we derive an alternative proof of the relationship between the edge toughness of a matroid and the fractional covering number of its dual matroid.

%This provides an alternative proof of the relationship between the edge toughness of a matroid and the fractional covering number of its dual. This result can then be applied to graphic matroids, yielding a formula for the edge toughness of a graph.

\paragraph{Case study 2: Vertex cover with budget.}
We further investigate the problem \textsc{Vertex Cover with Budget (VCB)}. 
First, using our LP complementation results we relate the ``time per budget'' ratio $\Fractional{t}$ to the fractional chromatic number $\Fractional{\chi}$ of the graph by $\frac{ 1 }{ \Fractional{t} } + \frac{ 1 }{ \Fractional{\chi} } = 1$. 
Second, we show that, surprisingly, for any finite budget we can also relate the optimal time with multicolourings of the graph. Finally, we prove that, computing an optimum solution, where the budget is finite, is NP-complete.

\medskip

The rest of the paper is organised as follows. Section \ref{section:LP_complementation} first gives the LP complementation theorem. It then investigates its consequences to IP and derives the $\alpha$ and $\beta$ bounds. In Section \ref{section:hypergraphs}, we introduce the complement of a hypergraph and apply the LP complementation theorem to obtain the hypergraph complementation theorem. In Section \ref{section:applications}, we apply the hypergraph complementation theorem to obtain general results on the fractional dominating number of digraphs and to obtain a new proof of a result on the edge toughness of matroids. Finally, Section \ref{section:budget} applies our results from Sections \ref{section:LP_complementation} and \ref{section:hypergraphs} to the \textsc{VCB} problem.

\section{Linear Programming complementation} \label{section:LP_complementation}

\subsection{The LP complementation theorem} \label{section:complementation_theorem}

For any linear program (LP) $R$ which is feasible and bounded, we denote its optimal value as $\Opt(R)$. If $P$ is a maximisation problem, then we denote $\Opt(P) = -\infty$ if $P$ is infeasible and $\Opt(P) = \infty$ if $P$ is unbounded. Similarly, if $Q$ is a minimisation problem, then we denote $\Opt(Q) = \infty$ if $Q$ is infeasible and $\Opt(Q) = -\infty$ if $Q$ is unbounded.  We denote the all-zero vector or matrix as $0$, regardless its dimension; similarly, the all-ones vector or matrix is denoted as $1$.

We define the \Define{complement} of an LP $R$, which we denote $\LPComplement{R}$, as follows. Let $c \in \R^n$, $b \in \R^m$, $A \in \R^{m \times n}$, then for the following maximisation LP $P$, we have
\begin{align*}
    P &:~ \max \{ c^\top x : Ax \le b \},\\
    \LPComplement{P} &:~ \min \{ c^\top x : (b c^\top - A) x \ge b \}.
\end{align*}
% \begin{alignat*}{3}
% 	P: &\qquad&
% 	\max		&\;&	c^\top x	\\
% 	&&\text{s.t.}	&\;&	Ax 	&\le b\\
% 	&&			&\;&	x & \ge 0,\\
% 	\\
% 	\LPComplement{P}:&\qquad&
% 	\min		&\;&	c^\top x	\\
% 	&&\text{s.t.}	&\;&	(b c^\top -A)x &\ge b\\
% 	&&				&\;&	x & \ge 0.
% \end{alignat*}
Similarly, let $v \in \R^n$, $u \in \R^m$, $M \in \R^{m \times n}$, then for the following minimisation LP $Q$, we have
\begin{align*}
    Q &:~ \min \{ v^\top x : Mx \ge u \},\\
    \LPComplement{Q} &:~ \max \{ v^\top x : (u v^\top - M) x \le u \}.
\end{align*}
% \begin{alignat*}{3}
% 	Q: &\qquad&
% 	\min		&\;&	v^\top x	\\
% 	&&\text{s.t.}	&\;&	Mx 	&\ge u\\
% 	&&			&\;&	x & \ge 0,\\
% 	\\
% 	\LPComplement{Q}: &\qquad&
% 	\max		&\;&	v^\top x	\\
% 	&&\text{s.t.}	&\;&	(u v^\top - M)x &\le u\\
% 	&&			&\;&	x & \ge 0.
% \end{alignat*}
The definition above is extended to general LPs in Table \ref{table:LPComplement}.

\begin{table}[H]
\begin{center}
\def\arraystretch{1.5}
\begin{tabular}{l|l}
	Primal $R$ 				& Complement $\LPComplement{R}$						\\
	\hline
	$\max c^\top x$		& $\min c^\top x$					\\
	$A_l x \le b_l$ 	& $(b_l c^\top - A_l) x \ge b_l$	\\
	$A_e x = b_e$ 		& $(b_e c^\top - A_e) x = b_e$		\\
	$A_g x \ge b_g$ 	& $(b_g c^\top - A_g) x \le b_g$	\\
	$x_l \le 0$ 		& $x_l \le 0$						\\
	$x_g \ge 0$ 		& $x_g \ge 0$						\\
	$x_f \text{ free}$	& $x_f \text{ free}$
\end{tabular}
\end{center}
\caption{General definition of LP complement} \label{table:LPComplement}
\end{table}
%
%We define the \Define{complement} of an LP $R$, which we denote $\LPComplement{R}$, as follows. Let $c \in \R^n$, $b \in \R^m$, $A \in \R^{m \times n}$, then for the following maximisation LP $P$, we have
%\begin{alignat*}{3}
%	P: &\qquad&
%	\max		&\;&	c^\top x	\\
%	&&\text{s.t.}	&\;&	Ax 	&\le b\\
%	&&			&\;&	x & \ge 0,\\
%	\\
%	\LPComplement{P}:&\qquad&
%	\min		&\;&	c^\top x	\\
%	&&\text{s.t.}	&\;&	(b c^\top -A)x &\ge b\\
%	&&				&\;&	x & \ge 0.
%\end{alignat*}
%Similarly, let $v \in \R^n$, $u \in \R^m$, $M \in \R^{m \times n}$, then for the following minimisation LP $Q$, we have
%\begin{alignat*}{3}
%	Q: &\qquad&
%	\min		&\;&	v^\top x	\\
%	&&\text{s.t.}	&\;&	Mx 	&\ge u\\
%	&&			&\;&	x & \ge 0,\\
%	\\
%	\LPComplement{Q}: &\qquad&
%	\max		&\;&	v^\top x	\\
%	&&\text{s.t.}	&\;&	(u v^\top - M)x &\le u\\
%	&&			&\;&	x & \ge 0.
%\end{alignat*}
Complementation is an involution, i.e.~$R = \LPComplement{ \LPComplement{R} }$. Moreover, complementation commutes with duality: indeed, if $\LPDual{R}$ denotes the dual of $R$, then we have
$\LPComplement{ (\LPDual{R}) } = \LPDual{ \left( \LPComplement{R} \right) }$.

Say two real numbers $x, y > 1$ are a \Define{complement pair} if $\frac{1}{x} + \frac{1}{y} = 1$.  The main result is that, provided $1 < \Opt(R) < \infty$ or $1 < \Opt(\LPComplement{R}) < \infty$, then the optimal values of $R$ and $\LPComplement{R}$ form a complement pair.

\begin{theorem}[LP complementation theorem] \label{th:LP_complementation2}
For any LP $R$, $1 < \Opt(R) < \infty$ if and only if $1 < \Opt(\LPComplement{R}) < \infty$, in which case 
\[
	\frac{1}{\Opt(R)} + \frac{1}{ \Opt( \LPComplement{R} ) } = 1.
\]
\end{theorem}

\begin{proof}
Without loss of generality, let $P : \max\{ c^\top x : Ax \le b \}$. Suppose $1 < \Opt(P) < \infty$, say $\Opt(P) = 1 + a$ for some $a > 0$. Let $x$ be an optimal solution of $P$, and let $\Complement{x} = \frac{1}{a} x$. We then have 
\[
	(b c^\top - A) \Complement{x} = \frac{1 + a}{a} b - \frac{1}{a}A x \ge b,
\]
and hence $\Complement{x}$ is a feasible solution of $\LPComplement{P}$, with value $1 + \frac{1}{a}$.

We have just shown that $\LPComplement{P}$ has a feasible solution of value greater than one. We now prove that $\Opt(\LPComplement{P}) > 1$. For the sake of contradiction, suppose that $\LPComplement{P}$ has a feasible solution with value at most $1$, then for any $\epsilon > 0$, $\LPComplement{P}$ has a feasible solution $\Complement{y}$ with value $1 + \epsilon$. Let $y = \frac{1}{\epsilon} \Complement{y}$, then by the same reasoning as above, $y$ is a feasible solution of $P$ with value $1 + \frac{1}{\epsilon}$; we conclude that $P$ is unbounded, which is the desired contradiction.

Having established that $1 < \Opt(\LPComplement{P}) < \infty$, we find that the first paragraph showed that
\[
	\frac{1}{ \Opt(P) } + \frac{1}{ \Opt(\LPComplement{P}) } \ge \frac{1}{ a + 1 } + \frac{a}{ a + 1 } = 1.
\]
We now prove the reverse inequality. Let $\Opt(\LPComplement{P}) = 1 + \bar{a}$ with $\Complement{a} > 0$ and $\Complement{x}$ be an optimal solution of $\LPComplement{P}$. Then $x = \frac{1}{\Complement{a}} \Complement{x}$ is a feasible solution of $P$ with value $1 + \frac{1}{\Complement{a}}$, and we obtain
\[
	\frac{1}{ \Opt(P) } + \frac{1}{ \Opt(\LPComplement{P}) } \le \frac{\Complement{a}}{ \Complement{a} + 1 } + \frac{1}{ \Complement{a} + 1 } = 1.
\]

The case where we suppose $1 < \Opt(\LPComplement{P}) < \infty$ instead is similar and hence omitted.
\end{proof}

If $(a,b)$ form a complement pair, and say $a \le b$, then $a \le 2$ and $b \ge 2$. The LP complementation theorem then has this immediate consequence.

\begin{corollary} \label{corollary:2_separation}
Suppose $1 < \Opt(P) \le \Opt(\LPComplement{P}) < \infty$. Then 
\[
    \Opt(P) \le 2 \le \Opt( \LPComplement{P} ).
\]
Moreover, the following are equivalent: $\Opt(P) = 2$; $\Opt( \LPComplement{P} ) = 2$; $\Opt(P) = \Opt( \LPComplement{P} )$.
\end{corollary}

\subsection{Feasibility and boundedness} \label{section:feasibility}

The strong duality theorem not only states that the optimal values of a primal LP and that of its dual are equal whenever they are finite, but it also considers the case of infeasibility and unboundedness: an LP is infeasible if and only if its dual is unbounded. Duality hence considers three possible scenarios for a maximisation LP $P$: $\Opt(P) = -\infty$, $-\infty < \Opt(P) < \infty$, and $\Opt(P) = \infty$; then only three scenarios are possible for the primal-dual pair $(P, \LPDual{P})$.

% As we can see, the complementation theorem only considers feasible, bounded LPs with an optimal value greater than one. This condition is met by many natural LPs, as we shall see in the next section. Nonetheless, different scenarios can occur when one of the LPs is infeasible, unbounded, or with an optimal value less than or equal to one, as seen below. 

Complementation, on the other hand, considers four possible scenarios for $P$: $\Opt(P) = -\infty$, $-\infty < \Opt(P) \le 1$, $1 < \Opt(P) < \infty$ and $\Opt(P) = \infty$. So this could make up to sixteen scenarios for the primal-complement pair $(P, \LPComplement{P})$. The LP complementation theorem implies that if $1 < \Opt(P) < \infty$, then so does $\Opt(\LPComplement{P})$ and vice versa. The proof of Theorem \ref{th:LP_complementation2} also shows that if $\Opt(P) > 1$, then $\LPComplement{P}$ is feasible, i.e.~$\Opt(\LPComplement{P}) < \infty$. Therefore, if $\Opt(P) = \infty$, then $\Opt( \Complement{P} ) \le 1$. This leaves nine possible scenarios; for each of those we give an example in Table \ref{table:scenarios} below.

\begin{table}[H]
\begin{center}
\def\arraystretch{1.5}
\begin{tabular}{l | l | l | l}
    $P$ & $\LPComplement{P}$ & $\Opt(P)$ & $\Opt(\LPComplement{P})$\\
    \hline
    $\max\{ x: x \le b \}$ & $\min\{ x :  (b-1) x \ge b \}$ & $b > 1$ & $\frac{b}{b-1} > 1$ \\
    $\max\{ x : x \le 0, x \ge 1 \}$ & $\min\{ x : x \le 0 \}$ & $-\infty$ & $-\infty$ \\
    $\max\{ -x : x \ge 1, x \le 0 \}$ & $\min\{ -x : -2x \le 1, x \le 0 \}$ & $-\infty$ & $0$ \\
    $\max\{ x : x \le 1, x \ge 2 \}$ & $\min\{ x : 0 \ge 1 \}$  & $-\infty$ & $\infty$ \\
    $\max\{ -x : x \ge 0 \}$ & $\min\{ -x : x \ge 0\}$ & $0$ & $-\infty$ \\
    $\max\{ x : x = 0\}$ & $\min\{ x : x = 0 \}$ & $0$ & $0$ \\
    $\max\{ x : x \le 1 \}$ & $\min\{ x : 0 \ge 1 \}$ & $1$ & $\infty$ \\
    $\max\{ x : x \ge 2 \}$ & $\min\{ x: x \le 2 \}$ & $\infty$ & $-\infty$ \\
    $\max\{ x : x \ge 0 \}$ & $\min\{ x : x \ge 0 \}$ & $\infty$ & $0$ 
\end{tabular}
\caption{Examples of the nine possible scenarios for $( \Opt(P), \Opt(\LPComplement{P}) )$. Here $x$ is a single variable.} \label{table:scenarios}
\end{center}
\end{table}

\subsection{Consequence for integer programming} \label{section:IP}

The proof of Theorem \ref{th:LP_complementation2} actually shows that, whenever $\Opt(P) > 1$, $x$ is an optimal solution of $P$ if and only if $\frac{1}{\Opt(P) - 1} x$ is an optimal solution of $\LPComplement{P}$. This has a consequence for integer programming. 

For any linear program $R$, adding the constraint that the variables be integral yields an integer program, which we denote $\IP{R}$. We consider the LPs in standard form $P : \max\{ c^\top x : Ax \le b  \}$ and $Q : \min\{ v^\top x : Mx \ge u \}$. For any $s,t \in \N$, we then introduce
\begin{align*}
	P_s &:~ \max \{ c^\top x	: Ax \le sb \},\\
	%\\
	Q_t &:~ \min \{ v^\top x :	Mx \ge tu \}.
\end{align*}
Clearly, $\Opt(P_s) = s \Opt(P)$ and $\Opt(Q_t) = t \Opt(Q)$.

The LP complementation theorem has two consequences for IPs of the form $\IP{P_s}$ or $\IP{Q_t}$. We give these for $\IP{P_s}$ below; their counterparts for $\IP{Q_t}$ are analogous and hence omitted.

\begin{corollary} \label{corollary:IP}
Suppose $P : \max \{ c^\top x : Ax \le b \}$, where $A$, $b$, and $c$ are all integral. Let $s,t \in \N$ such that $\tilde{x} \in (\Z / s)^n$ is an optimal solution of $P$ with value $1 + \frac{t}{s} > 1$. Then
\begin{enumerate}
    \item \label{item:common_solution}
    The four optimisation problems $P_s$, $\IP{P_s}$, $(\LPComplement{P})_t$, and $\IP{ ( \LPComplement{P} )_t }$ all have a common optimal solution $\hat{x} = s \tilde{x}$ of value $s + t$.

    \item \label{item:Pst}
    We have $\Opt( \IP{P_{st}} ) = \Opt(P_{st})$ and $\Opt( \IP{ (\LPComplement{P})_{st}  } ) = \Opt( (\LPComplement{P})_{st} )$, thus 
    \[
        \frac{1}{ \Opt( \IP{P_{st}} ) } + \frac{1}{ \Opt( \IP{ (\LPComplement{P})_{st} } )  } = \frac{1}{st}.
    \]

\end{enumerate}
\end{corollary}

\begin{proof}
\begin{enumerate}
    \item 
    By definition, $\hat{x}$ is an optimal solution of $P_s$ with value $s+t$. Since
    $
        \hat{x} = t \frac{1}{\Opt(P) - 1} \tilde{x},
    $
    we obtain that $\hat{x}$ is also an optimal solution of $(\LPComplement{P})_t$. Moreover, $\hat{x}$ is integral, therefore it is also an optimal solution of $\IP{P_s}$ and $\IP{ (\LPComplement{P})_t }$.
    
    \item 
    It is easily seen that for any LP $R$ and any $a \in \N$,  if $R_a$ has an integral optimal solution, then so does $R_{ab}$ for any $b \in \N$. By item \ref{item:common_solution}, $P_s$ and $(\LPComplement{P})_t$ both have integral optimal solutions, thus so do $P_{st}$ and $(\LPComplement{P})_{st}$. Applying the LP complementation theorem then finishes the proof.

\end{enumerate}
\end{proof}

% \begin{corollary} \label{corollary:q+r}
% If $P = \max\{ c^\top x : Ax \le b \}$, with $c \in \{0,1\}^n$, then for all $q, r \in \N$,
% \[
%     \Opt( \IP{P_q} ) \ge q + r \iff \Opt( \IP{P_r} ) \ge q + r.
% \]
% \end{corollary}

\subsection{Bounds}

Let $P : \max\{c^\top x : Ax \le b\}$, where $b > 0$ and $A \ne 0$. Let $Q : \min\{ v^\top x: Mx \ge u \}$ with $u > 0$ and $M \ne 0$. Let $A_i$ and $M_i$ denote the $i$-th rows of $A$ and $M$, respectively. 
We define the \Define{rank function} of $P$ and $Q$, respectively by
\begin{align*}
    \rho_P(x)   &= \max \left\{ \frac{A_i x}{b_i} : 1 \le i \le m \right\} \\
    \sigma_Q(x) &= \min \left\{ \frac{M_i x}{u_i} : 1 \le i \le m \right\}.
\end{align*}
Then $x$ is a feasible solution of $P$ (of $Q$, respectively) if and only if $\rho_P(x) \le 1$ ($\sigma_Q(x) \ge 1$, respectively). We now introduce
\begin{align*}
    \alpha(P) &= \max \left\{ \frac{c^\top x}{ \rho_P(x) } :  \rho_P(x) < c^\top x \right\},\\
    \beta(Q) &= \min \left\{ \frac{v^\top x}{ \sigma_Q(x) } : 0 < \sigma_Q(x)  \right\}.
\end{align*}
The $\alpha(P)$ quantity is a refinement of the simple following observation. The vector $x'$ with $x'_i =  \frac{b_i}{\sum_j A_{ij}}$ for all $i$ is a feasible solution of $P$ with $\rho_P(x') = 1$ and value $c^\top x' = \frac{c^\top x'}{ \rho_P(x')} \le \alpha(P)$. A similar comment can be made about $\beta(Q)$.

We also introduce the counterparts for the IPs as
\begin{align*}
    \alpha(\IP{P}) &= \max \left\{ \frac{c^\top x}{ \rho_P(x) } :  \rho_P(x) < c^\top x , x \in \Z^n \right\},\\
    \beta(\IP{Q}) &= \min \left\{ \frac{v^\top x}{ \sigma_Q(x) } : 0 < \sigma_Q(x), x \in \Z^n \right\}.
\end{align*}

Suppose that $1 < \Opt( \IP{P} ) < \infty$ and $1 < \Opt( \IP{\LPComplement{P}} ) < \infty$. We prove that $\alpha(\IP{P})$ and $\beta(\IP{ \LPComplement{P} })$ are complement pairs. (The same is true for $\alpha(P)$ and $\beta( \LPComplement{P} )$, as we shall prove later.)

\begin{lemma} \label{lemma:alpha}
If $1 < \Opt( \IP{P} ) < \infty$ and $1 < \Opt( \IP{\LPComplement{P}} ) < \infty$, then we have
\[
    \frac{1}{\alpha(\IP{P})} + \frac{1}{\beta(\IP{ \LPComplement{P} })} = 1.
\]
\end{lemma}

\begin{proof}
By definition, we have $\rho_P(x) + \sigma_{\LPComplement{P}}(x) = c^\top x$. Therefore,
\begin{align*}
    1 - \frac{1}{\alpha(\IP{P})} &= 1 - \min\left\{ \frac{ \rho_P(x) }{ c^\top x } : \rho_P(x) < c^\top x, x \in \Z^n \right\} \\
    &= \max\left\{ \frac{ c^\top x - \rho_P(x) }{ c^\top x } : \rho_P(x) < c^\top x, x \in \Z^n \right\} \\
    % &= \frac{1}{ \min\left\{ \frac{ c^\top x }{ c^\top x - \rho_P(x) } : 0 < \rho_P(x) < c^\top x  \right\} } \\
    &= \frac{1}{ \min\left\{ \frac{ c^\top x }{ \sigma_{\LPComplement{P}}(x) } : 0 < \sigma_{\LPComplement{P}}(x) , x \in \Z^n \right\} } \\
    &= \frac{1}{\beta( \IP{ \LPComplement{P} })}.
\end{align*}
\end{proof}

We obtain the more complete version of Theorem \ref{theorem:alpha1} as follows.

\begin{theorem} \label{theorem:alpha}
Let $P : \max\{c^\top x : Ax \le b\}$, where $b > 0$ and $A \ne 0$. Let $Q : \min\{ v^\top x: Mx \ge u \}$ with $u > 0$ and $M \ne 0$. 
\begin{enumerate}
    \item 
    If $\Opt( \IP{P} ) > 1$, then $1 < \Opt( \IP{P} ) \le \alpha(\IP{P}) \le \alpha(P) = \Opt( P )$.

    \item
    If $\Opt( \IP{Q} ) > 1$, then $\Opt( Q ) = \beta( Q ) \le \beta( \IP{ Q } ) \le \Opt( \IP{ Q } )$.

    \item
    If $1 < \Opt(\IP{P}) \le \Opt( \IP{\LPComplement{P}} ) < \infty$, then $\Opt(P) = \alpha(\IP{P}) \iff \Opt(\LPComplement{P}) = \beta(\IP{ \LPComplement{P} })$.
\end{enumerate}
\end{theorem}

\begin{proof}
\begin{enumerate}
    \item
    We prove the bounds on $\alpha$.
    \begin{enumerate}
        \item \label{item:1} $\Opt( \IP{P} ) \le \alpha(\IP{P})$. Let $x'$ be an optimal solution of $\IP{P}$. Then $0 < \rho_P(x') \le 1 < c^\top x'$, thus
    \[
        \alpha(\IP{P}) \ge \frac{ c^\top x' }{ \rho_P(x') } \ge \Opt( \IP{P} ).
    \]
    
        \item $\alpha(\IP{P}) \le \alpha(P)$. By definition.
        
        \item $\alpha(P) \le \Opt(P)$. Let $x''$ be such that $\alpha(P) = \frac{ c^\top x'' }{ \rho_P(x'') }$, then let $y = \frac{1}{\rho_P(x'')} x''$. We have 
        \[
        Ay = \frac{1}{\rho_P(x'')} A x'' \le \frac{1}{\rho_P(x'')} (\rho_P(x'') b) = b,
        \]
        hence $y$ is a feasible solution of $P$; its value is $c^\top y = \alpha(P)$.
        
        \item $\Opt(P) \le \alpha(P)$. Same proof as item \ref{item:1} above. 
    \end{enumerate}

    \item
    Similar and hence omitted. 
    
    \item 
    The pair $( \Opt(P), \Opt(\LPComplement{P}) )$ is a complement pair by the LP complementation theorem, while $(\alpha( \IP{P} ), \beta( \IP{\LPComplement{P}} ))$ is a complement pair by Lemma \ref{lemma:alpha}. Therefore, $\Opt(P) = \alpha( \IP{P} )$ if and only if $\Opt(\LPComplement{P}) = \beta( \IP{\LPComplement{P}} )$.
\end{enumerate}

\end{proof}

\section{Fractional hypergraph theory} \label{section:hypergraphs}

\subsection{Fractional hypergraph parameters} \label{section:hypergraph_parameters}

Many important graph parameters, such as the clique number, chromatic number, matching number, etc. can be viewed as the optimal values of IPs defined on hypergraphs related to the original graph. Fractional hypergraph theory then lifts the integrality constraint and focuses on the fractional analogues of those parameters, which are the optimal values of the corresponding LP relaxations. In this section, we review four important fractional hypergraph parameters, and how they are related. A comprehensive account of those parameters can be found in \cite{SU97}.

A (finite) \Define{hypergraph} is a pair $H = (V,E)$, where $V$ is a set of $n$ vertices and $E$ is a multiset of $m$ edges, each being a subset of vertices. Recall the following concepts for a hypergraph $H$. Its \Define{incidence matrix} is $M=M_H \in \R^{n \times m}$ such that, for all $v \in V$ and $e \in E$,
\[
	M_{ve} = \begin{cases}
	1 &\text{if } v \in e\\
	0 & \text{otherwise}.
	\end{cases}
\]
A vertex is \Define{universal} if it belongs to all edges of $H$. On the other hand, a vertex is \Define{isolated} if it does not belong to any edge of $H$. Say an edge $e$ is \Define{complete} if $e = V$ and that it is \Define{empty} if $e = \emptyset$. 
% We say $H$ is \Define{nontrivial} if it has no empty edges, no complete edges, no universal vertices, and no isolated vertices.

\medskip

We now introduce four LPs related to a hypergraph $H$; we shall then apply the LP complementation theorem to them. All those LPs have an optimal value in $[1, \infty]$. Technically, if the optimal value is either $1$ or $\infty$, then the LP complementation theorem does not apply. However, we highlight these degenerate cases, which can easily be handled separately. By using the convention that $1$ and $\infty$ form a complement pair, we can then include these degenerate cases in our hypergraph complementation theorem.

A \Define{covering} of $H$ is a set of edges whose union is equal to $V$. The \Define{covering number} $k(H)$ of $H$ is the minimum size of a covering of $H$. The \Define{fractional covering number} $\Fractional{k}(H)$ of $H$ is the optimal value of the following LP, which we give in two forms: a concise matrix form and a more explicit form.
% \begin{alignat*}{7}
% 	K(H):	&\qquad&\min		&\;&	1^\top x&		&\qquad\qquad 	\min 	&\;& 	\sum_{e \in E} x_e\\
% 			&&		\text{s.t.}	&&		M_H x 	&\ge 1,	&			\text{s.t.} && 		\sum_{e \ni v} x_e 	&\ge 1 	&\;&\forall v \in V,\\
% 			&&					&&		x 		&\ge 0.	& 						&& 		x_e 				&\ge 0	&&	\forall e \in E.
% \end{alignat*}
\begin{align*}
    K(H) &:~ \min\{ 1^\top x :  M_H x \ge 1, x \ge 0\}\\
    &= \min 	\left\{	\sum_{e \in E} x_e : \sum_{e \ni v} x_e 	\ge 1 	\;\forall v \in V, x_e 				\ge 0	\;	\forall e \in E \right\}.
\end{align*}

It is easily seen that the covering number is actually the optimal value of $\IP{ K(H) }$. We remark that $K(H)$ is feasible if and only if $H$ has no isolated vertices. Clearly, if $K(H)$ is feasible, then it has an optimal solution. In that case, $\Fractional{k}(H) = \Opt(K(H)) \ge 1$, with strict inequality if and only if $H$ has no complete edges.

A \Define{packing} of $H$ is a set of vertices such that every edge contains at most one of those vertices. 
The \Define{packing number} $p(H)$ of $H$ is the maximum size of a packing of $H$.
The \Define{fractional packing number} $\Fractional{p}(H)$ of $H$ is the optimal value of the LP dual to $K(H)$:
% \begin{alignat*}{7}
% 	P(H) = \LPDual{K(H)}:	&\qquad&\max		&\;&	1^\top y	&		&\qquad\qquad 	\max 	&\;& 	\sum_{v \in V} y_v\\
% 						&&		\text{s.t.}	&&		M_H^\top y 	&\le 1,	&			\text{s.t.} && 		\sum_{v \in e} y_v 	&\le 1 	&\;&\forall e \in E,\\
% 						&&					&&		y 			&\ge 0.	& 						&& 		y_v 				&\ge 0	&&	\forall v \in V.
% \end{alignat*}
\begin{align*}
    P(H) = \LPDual{K(H)} &:~ \max	\{	1^\top y : 	M_H^\top y 	\le 1, y \ge 0\} \\
    &= \max \left\{ 	\sum_{v \in V} y_v : \sum_{v \in e} y_v \le 1 	\; \forall e \in E, y_v \ge 0	\; \forall v \in V \right\}.
\end{align*}
Again, the maximum size of a packing of $H$ corresponds to the optimal value of the analogous IP. We remark that $P(H)$ is always feasible. However, $P(H)$ is bounded if and and only if $H$ has no isolated vertices. In that case, $\Fractional{p}(H) = \Opt(P(H)) > 1$ if and only if it has no complete edges. LP duality then yields $\Fractional{p}(H) = \Fractional{k}(H)$.%, whenever $H$ has no isolated vertices.

For any hypergraph $H = (V, E)$, its \Define{dual} is $\HDual{H} = (E, \Dual{V})$, where $\Dual{V} = \{ \{ e : e \ni v\} : v \in V \}$. We then have $M_{\HDual{H}} = (M_H)^\top$ and $\HDual{(\HDual{H})} \cong H$. We note that $H$ has no empty edge if and only if $\HDual{H}$ has no isolated vertex, and vice versa. %Therefore, $H$ is nontrivial if and only if $\HDual{H}$ is nontrivial. 

A \Define{matching} of $H$ is a set of disjoint edges; it corresponds to a packing of $\HDual{H}$. The \Define{fractional matching number} is then $\Fractional{\mu}(H) = \Fractional{p}(\HDual{H})$, i.e.~the optimal value of:
% \begin{alignat*}{7}
% 	M(H) = P(\HDual{H}):	&\qquad&\max		&\;&	1^\top y	&		&\qquad\qquad 	\max 	&\;& 	\sum_{e \in E} y_e\\
% 						&&		\text{s.t.}	&&		M_H y 		&\le 1,	&			\text{s.t.} && 		\sum_{e \ni v} y_e 	&\le 1 	&\;&\forall v \in V,\\
% 						&&					&&		y 			&\ge 0.	& 						&& 		y_e 				&\ge 0	&&	\forall e \in E.
% \end{alignat*}
\begin{align*}
    M(H) = P(\HDual{H}) &:~ \max	\{	1^\top y : 	M_H y \le 1, y \ge 0\} \\
    &= \max 	\left\{ 	\sum_{e \in E} y_e : \sum_{e \ni v} y_e \le 1 	\; \forall v \in V, y_e \ge 0 \; \forall e \in E \right\}.
\end{align*}

A \Define{transversal} of $H$ is a set of vertices such that every edge contains a vertex from that set; it corresponds to a covering of $\HDual{H}$. The \Define{fractional transversal number} is then $\Fractional{\tau}(H) = \Fractional{k}(\HDual{H})$, i.e.~the optimal value of:
% \begin{alignat*}{7}
% 	T(H) = K(\HDual{H}):	&\qquad&\min		&\;&	1^\top x	&		&\qquad\qquad 	\min 	&\;& 	\sum_{v \in V} x_v\\
% 						&&		\text{s.t.}	&&		M_H^\top x 	&\ge 1,	&			\text{s.t.} && 		\sum_{v \in e} x_v 	&\ge 1 	&\;&\forall e \in E,\\
% 						&&					&&		x 			&\ge 0.	& 						&& 		x_v 				&\ge 0	&&	\forall v \in V.
% \end{alignat*}
\begin{align*}
    T(H) = K(\HDual{H}) &:~ \min	\{	1^\top x : M_H^\top x \ge 1, x \ge 0 \} \\
    &= \min \left\{	\sum_{v \in V}  x_v : \sum_{v \in e} x_v \ge 1 	\; \forall e \in E,  x_v \ge 0	\;\forall v \in V \right\}.
\end{align*}
Again, LP duality yields $\Fractional{\mu}(H) = \Fractional{\tau}(H)$.%, whenever $H$ has no empty edges.

In summary, for any $H$ we have $\Fractional{\tau}(H) = \Fractional{k}(\HDual{H}) = \Fractional{p}(\HDual{H}) = \Fractional{\mu}(H)$.

\subsection{Hypergraph complementation} \label{Section:hypergraph_complementation}

We define the \Define{complement} of $H$ as $\HComplement{H} = (V, \{V \setminus e : e \in E \})$. We then have $M_{\HComplement{H}} = 1 - M_H$. Hypergraph complementation is an involution that commutes with duality: $\HComplement{\HComplement{H}} = H$ and $	\HComplement{ \left( \HDual{H} \right) } = \HDual{ \left( \HComplement{H} \right) }$.

% The hypergraph $H$ is nontrivial if and only if its complement $\HComplement{H}$ is nontrivial.  Again, if $H$ is nontrivial, then
% \[
% 	\Fractional{\tau}(\HDual{\HComplement{H}}) = \Fractional{k}(\HComplement{H}) = \Fractional{p}(\HComplement{H}) = \Fractional{\mu}(\HDual{\HComplement{H}}) > 1.
% \]
It can be easily shown that
\[
    \LPComplement{K( \HDual{H} )} = \LPDual{ K( \HComplement{H} ) }.
\]
Therefore, we obtain eight LPs, which are related in Figure \ref{fig:8LPs-2}.

For any $S \subseteq V$, let
\begin{align*}
	\rho_H(S)	&= \max \left\{ |S \cap e| : e \in E \right\},\\
	\alpha(H) &= \max \left\{ \frac{ |S| }{ \rho_H(S) } : S \subseteq V, \rho_H(S) > 0 \right\}.
\end{align*}
We similarly define for any $Z \subseteq E$
\begin{align*}
	\sigma_H(Z) 	&= \min \left\{ |\{ e \in Z: v \in e \}| : v \in V \right\},	\\
	\beta(H) 			&= \min \left\{  \frac{ |Z| }{ \sigma_H(Z)} : Z \subseteq E, \sigma_H(Z) > 0 \right\}.
\end{align*}
We immediately recognise that $\alpha(H) = \alpha( \IP{P(H)} )$ and $\beta(H) = \beta( \IP{K(H)} )$. Denoting the maximum size of an edge in $H$ as $\epsilon(H) = \max\{ |e| : e \in E \}$ and the minimum degree of a vertex in $H$ as $\delta(H) = \min\{ |e \ni v| : v \in V \}$, we have
\[
    \alpha(H) \ge \frac{|V|}{ \epsilon(H) }, \quad \beta(H) \le \frac{ |E| }{ \delta(H) }.
\]

% %%%%% BEGIN Gamma stuff. Previously used for some stuff with matroids but obsolete now

% We further define
% \[
% 	\gamma(H) = \min \left\{ \frac{ |T| }{ |T| - \rho_H(T) } : T \subseteq V, |T| > \rho_H(T) \right\}.
% \]

% \begin{lemma} \label{lemma:beta_gamma}
% For any hypergraph $H$, $\beta(\HDual{H}) = \gamma(\HComplement{H})$.
% \end{lemma}

% \begin{proof}
% We denote the set of edges of $H$ as $E$, and the set of edges of $\HComplement{H}$ as $\Complement{E}$. For any $T \subseteq V$, we have
% \[
% 	\tilde{\rho}_{\HDual{H}}(T) = \min \{ |T \cap e| : e \in E \} = |T| - \max \{ |T \cap \Complement{e}| : \Complement{e} \in \Complement{E} \} = |T| - \rho_{\HComplement{H}}(T),
% \]
% and hence
% \begin{align*}
% 	\beta( \HDual{H} ) &= \min \left\{ \frac{ |T| }{ \tilde{\rho}_{\HDual{H}} (T) } : T \subseteq V, \tilde{\rho}_{\HDual{H}} (T) > 0 \right\} \\
% 	& = \min \left\{ \frac{ |T| }{ |T| - \rho_{\HComplement{H}}(T) } : {T \subseteq V, |T| > \rho_{\HComplement{H}}(T) > 0 } \right\}\\
% 	& = \gamma(\HComplement{H}).
% \end{align*}
% \end{proof}

% %%%%% END Gamma stuff. Previously used for some stuff with matroids but obsolete now

The next theorem is a more complete version of Theorem~\ref{theorem:hypergraph_complementation-small}.

\begin{theorem}[Hypergraph complementation theorem] \label{theorem:hypergraph_complementation}
For any hypergraph $H$,
\[
	\frac{1}{\Fractional{k}(\HDual{H})} + \frac{1}{\Fractional{k}(\HComplement{H})} = 1.
\]
Moreover, we have the bounds
\[
    p(H) \le \alpha(H) \le \Fractional{k}(H) \le \beta(H) \le k(H),
\]
with equalities reached as follows:
\[
    \Fractional{k}(\HComplement{H}) = \alpha(\HComplement{H}) \iff \Fractional{k}(\HDual{H}) = \beta(\HDual{H}).
\]
\end{theorem}

\begin{proof}
We have $\LPComplement{ P( \HDual{H} ) } = K( \HComplement{H} )$. Theorem \ref{th:LP_complementation2} then shows that $\Fractional{p}(\HDual{H}) = \Fractional{k}(\HDual{H})$ and $\Fractional{k}( \HComplement{H} )$ are complement pairs. Theorem \ref{theorem:alpha} then gives the other two equations.
\end{proof}

Obviously, the hypergraph complementation theorem holds for all four parameters reviewed in Section \ref{section:hypergraph_parameters}.

\begin{corollary}
For any hypergraph $H$,
\[
	\frac{1}{ \Fractional{k}(\HDual{H}) } + \frac{1}{ \Fractional{k}(\HComplement{H}) } = 
	\frac{1}{ \Fractional{p}(\HDual{H}) } + \frac{1}{ \Fractional{p}(\HComplement{H}) } = 
	\frac{1}{ \Fractional{\mu}(\HDual{H}) } + \frac{1}{ \Fractional{\mu}(\HComplement{H}) } = 
	\frac{1}{ \Fractional{\tau}(\HDual{H}) } + \frac{1}{ \Fractional{\tau}(\HComplement{H}) } = 1.
\]
\end{corollary}

\section{Applications to fractional graph theory} \label{section:applications}

\subsection{Fractional domination in graphs and digraphs} \label{section:fractional_domination}

All the digraphs we consider are simple (no parallel arcs) and irreflexive (no loops). Thus, a \Define{digraph} is a pair $D = (V(D), E(D))$, where $E(D) \subseteq V(D)^2 \setminus \{ (v,v) : v \in V(D) \}$. The \Define{adjacency matrix} of $D$ is the $\{0,1\}$-matrix $A_D = (a_{ij} : i,j \in V(D))$, where $a_{ij} = 1$ if and only if $(i,j) \in E(D)$. We define \Define{digraph complement} of $D$, denoted $\DComplement{D}$, with $V(\DComplement{D}) = V(D)$ and $E(\DComplement{D}) = (V(D)^2 \setminus \{ (v,v) : v \in V(D) \}) \setminus E(D)$.

For any $v \in V(D)$, the \Define{open in-neighbourhood} of $v$ is $\InOpen{N}(v) = \{ u : (u,v) \in E(D) \}$; the \Define{closed in-neighbourhood} of $v$ is $\InClosed{N}(v) = \InOpen{N}(v) \cup \{v\}$. We thus define two hypergraphs $\InOpen{H}(D)$ and $\InClosed{H}(D)$, both with vertex set $V(D)$, and where the edges of $\InOpen{H}(D)$ are the open in-neighbourhoods of all vertices and the edges of $\InClosed{H}(D)$ are the closed in-neighbourhoods instead. Open and closed out-neighbourhoods are defined similarly, and hence we define $\OutOpen{H}(D)$ and $\OutClosed{H}(D)$ similarly as well. We note that $M_{\InOpen{H}(D)} = A_D$ and  $M_{\InClosed{H}(D)} = I_n + A_D$, where $n$ is the number of vertices in $D$ and $I_n$ is the identity matrix of size $n$. We then have
\[
\HDual{\InOpen{H}(D)} \cong \OutOpen{H}(D), \quad \HDual{\InClosed{H}(D)} \cong \OutClosed{H}(D), \quad  \HComplement{\OutOpen{H}(D)} = \OutClosed{H}(\DComplement{D}), \quad \HComplement{\OutClosed{H}(D)} = \OutOpen{H}(\DComplement{D}).
\]
An \Define{in-dominating set} of $D$ is a set $S$ of vertices such that for any $v \in V(D)$, there exists $s \in S \cap \InClosed{N}(v)$; in other words, it is a transversal of $\InClosed{H}(D)$. Similarly, a \Define{total in-dominating set} of $D$ is a transversal of $\InOpen{H}(D)$. We note that $D$ always has an in-dominating set ($V(D)$ itself), while $D$ has a total in-dominating set if and only if it has no sources (vertices with empty in-neighbourhoods). Out-dominating and total out-dominating sets are defined similarly. See the book by Haynes, Hedetniemi and Slater for a comprehensive survey of domination problems \cite{HHS98}.

The \Define{fractional in-dominating number} of $D$ and the \Define{fractional total out-dominating number} of $\DComplement{D}$ are then, respectively:
\begin{align*}
	\Fractional{\gammaIn}(D) 		&= \Fractional{\tau}(\InClosed{H}(D)) = \Fractional{\tau}(\HDual{\OutClosed{H}(D)}),\\
	\Fractional{\GammaOut}(\DComplement{D}) &= \Fractional{\tau}(\OutOpen{H}(\DComplement{D})) = \Fractional{\tau}(\HComplement{\OutClosed{H}(D)}).
\end{align*}
Let us call a vertex $v$ \Define{in-universal} in $D$ if $v \in \InClosed{N}(u)$ for all $u \in V$, i.e.~$v$ is a universal vertex of $\InClosed{H}(D)$. We note that $\Fractional{\gammaIn}(D) > 1$ if and only if $D$ has no in-universal vertices; the latter is also equivalent to $\Fractional{\GammaOut}(\DComplement{D}) < \infty$. We obtain the following; again the degenerate case of an in-universal vertex is handled by the $(1, \infty)$ complement pair.

\begin{theorem}[Domination complementation theorem] \label{th:domination2}
For any digraph $D$,
\[
	\frac{ 1 } { \Fractional{\gammaIn}(D) } + \frac{ 1 } { \Fractional{\GammaOut}(\DComplement{D}) } = 1.
\]
\end{theorem}

We focus on three special cases of Theorem \ref{th:domination2}. Firstly, a \Define{graph} $G$ is a symmetric digraph, i.e.~$A_G = A_G^\top$. For a graph $G$, in-neighbourhoods and out-neighbourhoods coincide. We then refer to $\Fractional{\gamma}(G) = \Fractional{\gammaIn}(G) = \Fractional{\gammaOut}(G)$ as the fractional dominating number of $G$; the fractional total dominating number of $G$ is defined and denoted similarly. 

\begin{corollary}
For any graph $G$,
\[
	\frac{ 1 } { \Fractional{\gamma}(G) } + \frac{ 1 } { \Fractional{\Gamma}(\DComplement{G}) } = 1.
\]
\end{corollary}

Secondly, a \Define{tournament} $T$ is a digraph where $(i,j) \in E(T)$ if and only if $(j,i) \notin E(T)$. If $T$ is a tournament, then $\DComplement{T}$ is obtained by reversing the direction of every arc in $T$. Thus, $\OutOpen{H}(\DComplement{T}) = \InOpen{H}(T)$ and we obtain the following corollary.

\begin{corollary}
For any tournament $T$,
\[
	\frac{1}{\Fractional{\gammaIn}(T)} + \frac{1}{ \Fractional{\GammaIn}(T) } = 1.
\]
In particular, $\Fractional{\gammaIn}(T) \le 2 \le \Fractional{\GammaIn}(T)$.
\end{corollary}

Thirdly, $D$ is \Define{$k$-regular} if for every vertex $v \in V(D)$, $|\InOpen{N}(v)| = |\OutOpen{N}(v)| = k$. Clearly, if $D$ has $n$ vertices, then $D$ is $k$-regular if and only if $\DComplement{D}$ is $(n-1-k)$-regular. The following result is a generalisation of the result in \cite[Theorem 7.4.1]{SU97}, which only applies to regular \emph{graphs}.

\begin{corollary}
For any $k$-regular digraph $D$ on $n$ vertices,
\[
    \Fractional{\gammaIn}(D) = \frac{n}{k+1}, \qquad \Fractional{\GammaOut}(D) = \frac{n}{k}.
\]
\end{corollary}

\begin{proof}
The value $n/(k+1)$ is an obvious upper bound for $\Fractional{\gammaIn}(D)$ (assign $1/(k+1)$ to each vertex); similarly, $n/(n - k - 1)$ is an upper bound for $\Fractional{\GammaOut}(\DComplement{D})$. By Theorem \ref{th:domination2}, these bounds must be tight.
\end{proof}

\subsection{Application to edge toughness of matroids} \label{section:matroids}

%The following lemma generalises the formula for the rank function of a dual matroid (see below). The proof is the same as the matroid case (see ), and hence is omitted.
%
%\begin{lemma}
%If $H$ is an $r$-uniform hypergraph, then for any $S \subseteq V$
%\[
%	\rho_{\HComplement{H}}(S) = |S| - r + \rho_H(V \setminus S).
%\]
%\end{lemma}

Let $M = (V, I)$ be a matroid \cite{Oxl06}, where $I$ is the collection of independent sets of $M$. A basis of $M$ is a maximal independent set. We then denote the set of bases of $M$ as $B(M)$ and we construct the hypergraph $H_B(M) = (V, B(M))$. The rank function of $M$ is then $\rho_M = \rho_{H_B(M)}$, i.e.~$\rho_M(S) = \max\{ S \cap e : e \in B(M) \}$. The dual matroid $\Complement{M}$ is then defined as $H_B(\Complement{M}) = \HComplement{H_B(M)}$. We note that the dual of a matroid is commonly denoted as $\Dual{M}$, but in this paper, denoting it as $\Complement{M}$ better reflects that its definition is in terms of hypergraph complementation, instead of hypergraph duality.

The \Define{edge toughness} (or strength) of $M$ is \cite{SU97}
\[
	\sigma'(M) = \min \left\{ \frac{|V \setminus S|}{ \rho_M(V) - \rho_M(S) } : S \subseteq V, \rho_M(V) > \rho_M(S) \right\}.
\]
The edge toughness is well defined unless $\rho_M(V) = 0$. Moreover, $\sigma'(M) = 1$ if and only if $M$ has a coloop, i.e.~an element $v$ that belongs to all bases. Say that $M$ is nontrivial if it falls in neither case mentioned above; then its edge toughness satisfies $\sigma'(M) > 1$. 

Next, we use the hypergraph complementation theorem to show that
the fractional transversal number and fractional matching number of a matroid coincide with its edge toughness.

\begin{theorem} \label{th:matroids}
For any nontrivial matroid $M$, we have
\[
	\Fractional{\mu}(H_B(M)) = \Fractional{\tau}(H_B(M)) = \sigma'(M).
\]
\end{theorem}

The proof of Theorem \ref{th:matroids} is based on the following lemma. For any hypergraph $H$, let
\[
	\gamma(H) = \min \left\{ \frac{ |T| }{ |T| - \rho_H(T) } : T \subseteq V, |T| > \rho_H(T) \right\}.
\]

\begin{lemma} \label{lemma:beta_gamma}
For any hypergraph $H$, $\beta(\HDual{H}) = \gamma(\HComplement{H})$.
\end{lemma}

\begin{proof}
We denote the set of edges of $H$ as $E$, and the set of edges of $\HComplement{H}$ as $\Complement{E}$. For any $T \subseteq V$, we have
\[
	\sigma_{\HDual{H}}(T) = \min \{ |T \cap e| : e \in E \} = |T| - \max \{ |T \cap \Complement{e}| : \Complement{e} \in \Complement{E} \} = |T| - \rho_{\HComplement{H}}(T),
\]
and hence
\begin{align*}
	\beta( \HDual{H} ) &= \min \left\{ \frac{ |T| }{ \sigma_{\HDual{H}} (T) } : T \subseteq V, \sigma_{\HDual{H}} (T) > 0 \right\} \\
	& = \min \left\{ \frac{ |T| }{ |T| - \rho_{\HComplement{H}}(T) } : {T \subseteq V, |T| > \rho_{\HComplement{H}}(T) } \right\}\\
	& = \gamma(\HComplement{H}).
\end{align*}
\end{proof}

\begin{proof}[Proof of Theorem \ref{th:matroids}]
Firstly, by the matroid base covering theorem (see \cite[Theorem 5.4.1]{SU97} or \cite[Corollary 42.1c]{Sch03}), the fractional covering number of a matroid reaches the $\alpha$ bound in Theorem \ref{theorem:hypergraph_complementation}. For the dual matroid, we obtain
\[
	\Fractional{k}( H_B( \Complement{M} ) ) = \alpha( H_B( \Complement{M} ) ).
\]
Moreover, thanks to Lemma \ref{lemma:beta_gamma}, we recognise that
\[
    \sigma'(M) = \gamma( H_B(\Complement{M}) ) = \beta( \HDual{ H_B(M) } ).
\]
Applying Theorem \ref{theorem:hypergraph_complementation} then yields
\[
    \Fractional{\mu}( H_B(M) ) = \Fractional{k}( \HDual{ H_B( M ) } ) = \beta( \HDual{ H_B( M ) } ) = \sigma'(M).
\]
\end{proof}

Applying the hypergraph complementation theorem, we obtain the following corollary, already given in \cite{SU97}.

\begin{corollary}[Theorem 5.6.8 in \cite{SU97}]
For any nontrivial matroid, we have
\[
    \frac{1}{ \sigma'(M) } + \frac{1}{ \Fractional{k}( H_B(\Complement{M}) ) } = 1.
\]
\end{corollary}

In particular, if $M_G$ is the cycle matroid of a graph $G$, where the elements of $M_G$ are the edges of $G$ and the bases of $M_G$ are all spanning forests of $G$ \cite{Oxl06}, then the edge toughness of $M_G$ reduces to the edge toughness (a.k.a strength) of $G$, defined as follows. For any $Z \subseteq E(G)$, let $G - Z$ denote the graph obtained by removing the edges from $Z$, and let $c(G - Z)$ denote the number of its connected components. Then
\[
	\sigma'(G) = \min \left\{ \frac{|Z|}{c(G - Z) - c(G)} : Z \subseteq E(G), c(G-Z) > c(G) \right\}.
\]
We remark that $\sigma'(G)$ is well defined if and only if $E(G)$ is nonempty. Moreover, $\sigma'(G) = 1$ if and only if $G$ has a cut edge, i.e.~$G$ has a connected component that is not $2$-edge connected.

Denote $H_{SF}(G) = H_B(M_G)$. The matching number of $H_{SF}(G)$ is the maximum number of edge-disjoint spanning forests in $G$. On the other hand, the transversal number of $H_{SF}(G)$ is the smallest size of an edge cut set of $G$. In particular, these two quantities are equal to $1$ whenever $G$ has a cut edge. Their fractional analogues are then equal to the edge toughness of $G$.

\begin{corollary}
For any graph $G$ whose connected components are all $2$-edge connected,
\[
	\Fractional{\mu}(H_{SF}(G)) = \Fractional{\tau}( H_{SF}(G) ) = \sigma'(G).
\] 
\end{corollary}

% \MGCOMM{(This is similar to Schrijver's book, Corollary 42.1.d or Theorem 42.3. Also see Chapter 51, p.891.) See [4, Corollary 5.3.3] which is [Schrijver, Theorem 42.2] with unit costs.}

\section{Vertex cover with budget} \label{section:budget}

% In this section, we study the problem of vertex cover with budget. First, we look at the vertex cover hypergraph, as it will dictate the asymptotics. Second, we look at finite budgets. In each case, we give exact results, bounds, and complexity results.

\subsection{The vertex cover hypergraph}

Let $G$ be a graph. A \Define{vertex cover} is a set $S$ of vertices such that $V \setminus S$ is independent. We define $H_{VC}(G)$ as the hypergraph whose edges are all the vertex covers of $G$. Then its complement is $\HComplement{H_{VC}(G)} = H_{IS}(G)$, whose edges are the independent sets of $G$. It immediately follows that $\Fractional{k}(\HComplement{H_{VC}(G)})$ is equal to $\Fractional{\chi}(G)$, the \Define{fractional chromatic number} of $G$. We then denote 
\[
    \Fractional{t}(G) = \Fractional{\mu}(H_{VC}(G)) = \Fractional{k}(\HDual{H_{VC}(G)}). 
\]
We have $\Fractional{\chi}(G) = 1$ if and only if $G$ is empty, in which case $\Fractional{t}(G) = \infty$. If $G$ is nonempty, then $\Fractional{\chi}(G) \ge 2$, with equality if and only if $G$ is bipartite. The hypergraph complementation theorem then yields
% \begin{theorem}\label{th:chromatic}
% For any nonempty graph $G$,
\[
	\frac{ 1 } { \Fractional{t}(G) } + \frac{ 1 } { \Fractional{\chi}(G) } = 1.
\]
% \end{theorem}

%A lot is known about the fractional chromatic number, which  we can readily use for $\Fractional{t}(G)$. 

Let us give some properties of the $\Fractional{t}(G) = \frac{\Fractional{\chi}(G)}{\Fractional{\chi}(G) - 1}$ quantity.
\begin{description}
    \item[Bounds]
%     Bounds on $\Fractional{\chi}(G)$ immediately translate to bounds on $\Fractional{t}(G)$.  Let $\alpha(G)$ denote the independence number of $G$, $\omega(G)$ denote its clique number and $\chi(G)$ denote its chromatic number. The following bounds can be found in \cite[Chapter 3]{SU97} or \cite{Lov75}.
% \begin{alignat*}{3}
%     \Fractional{\chi}(G) &\ge \frac{ n }{ \alpha(G) }: &\qquad \Fractional{t}(G) &\le \frac{n}{n - \alpha(G)}, \\ 
%  %
% 	\Fractional{\chi}(G) &\ge \frac{\chi(G)}{1 + \ln \alpha(G)}: &\qquad \Fractional{t}(G) &\le \frac{\chi(G)}{\chi(G) - 1 - \ln \alpha(G)},\\
% %
%     \Fractional{\chi}(G) &\ge \omega(G): &\qquad \Fractional{t}(G) &\le \frac{\omega(G)}{\omega(G) - 1},\\
% %
%     \Fractional{\chi}(G) &\le \chi(G):  &\qquad  \Fractional{t}(G) &\ge \frac{\chi(G)}{\chi(G) - 1}.
% \end{alignat*}

    Let $\alpha(G)$ denote the independence number of $G$, $\omega(G)$ denote its clique number and $\chi(G)$ denote its chromatic number. Then the bounds on $\Fractional{\chi}(G)$ in \cite[Chapter 3]{SU97} and \cite{Lov75}, given on the left hand side below, immediately translate to bounds on $\Fractional{t}(G)$, given on the right hand side below. 
\begin{alignat*}{3}
    \Fractional{\chi}(G) &\ge \frac{ n }{ \alpha(G) } &\qquad \longrightarrow \qquad \Fractional{t}(G) &\le \frac{n}{n - \alpha(G)}, \\ 
	\Fractional{\chi}(G) &\ge \frac{\chi(G)}{1 + \ln \alpha(G)} &\qquad  \longrightarrow\qquad \Fractional{t}(G) &\le \frac{\chi(G)}{\chi(G) - 1 - \ln \alpha(G)},\\
    \Fractional{\chi}(G) &\ge \omega(G) &\qquad \longrightarrow \qquad \Fractional{t}(G) &\le \frac{\omega(G)}{\omega(G) - 1},\\
    \Fractional{\chi}(G) &\le \chi(G)  &\qquad \longrightarrow \qquad  \Fractional{t}(G) &\ge \frac{\chi(G)}{\chi(G) - 1}.
\end{alignat*}

    \item[Possible values]
    If $G$ is non-empty, then $\Fractional{t}(G)$ is a rational number in $(1, 2]$. Conversely, for any rational number $q \in (1,2]$, there is $G$ with $\Fractional{t}(G) = q$ (since $\Fractional{\chi}(K(n,r)) = n/r$ for the Kneser graph with $n \ge 2r$ (see e.g. \cite{SU97})).
    
    \item[Complexity]
    Again, complexity results for $\Fractional{\chi}(G)$ can be converted into complexity results for $\Fractional{t}(G)$. Thus, for any $1 < s < 2$, determining whether $\Fractional{t}(G) \ge s$ is NP-complete (an immediate consequence of \cite{GLS81}). On the other hand, $\Fractional{t}(G)$ can be computed in polynomial time if $G$ is a line graph (see \cite[Section 4.5]{SU97}), or if $G$ is perfect (since the chromatic and fractional chromatic numbers coincide in that case). 
\end{description}

% Let $\alpha(G)$ denote the independence number of $G$ and $\chi(G)$ denote its chromatic number. Then since $\Fractional{\chi}(G) \ge n/\alpha(G)$ and $\Fractional{\chi}(G) \le \chi(G)$, we obtain
% \[
% 	\frac{\chi(G)}{\chi(G) - 1} \le \Fractional{t}(G) \le \frac{n}{n - \alpha(G)}.
% \]
% Conversely, since
% \[
% 	\Fractional{\chi}(G) \ge \frac{\chi(G)}{1 + \ln \alpha(G)},
% \]
% we obtain
% \[
% 	\Fractional{t}(G) \le \frac{\chi(G)}{\chi(G) - 1 - \ln \alpha(G)}.
% \]
% Moreover, if $\omega(G)$ denotes the clique number of $G$, we have
% \[
% 	\Fractional{t}(G) \le \frac{\omega(G)}{\omega(G) - 1}.
% \]

\subsection{Vertex cover with finite budget}

The \textsc{Vertex Cover with Budget (VSC)} problem is defined as follows. Let $G$ be a graph and $b$ a positive integer. For any family of $t$ vertex covers $S = \{S_1, \dots, S_t\}$ of $G$, we refer to the budget of $S$ as the maximum number of times a particular vertex appears in $S$:
\[
	\max\{ |\{ i : v \in S_i\} | : v \in V\}.
\]
For any $b \ge 1$, we denote the cardinality of the largest family of vertex covers with budget at most $b$ as $t_b(G)$. The problem is, given $G$ and $b$, to determine $t_b(G)$.

A \Define{$b$-fold matching} of a hypergraph $H$ is a set of edges of $H$ such that every vertex is contained in at most $b$ edges (so that a matching is a $1$-fold matching). The maximum size of a $b$-fold matching is denoted as $\mu_b(H)$. We immediately obtain that $t_b(G) = \mu_b( H_{VC}(G) )$.  Similarly, a $c$-fold covering is a set of edges of $H$ such that every vertex is contained in at least $c$ edges. The smallest size of a $c$-covering of $H$ is denoted as $k_c(H)$. We then have
\begin{align*}
    \Fractional{\mu}(H) &= \lim_{b \to \infty} \frac{ \mu_b(H) }{ b } = \max_{b \in \N} \frac{ \mu_b(H) }{ b },\\
    \Fractional{k}(H) &= \lim_{c \to \infty} \frac{ k_c(H) }{ c } = \max_{c \in \N} \frac{ k_c(H) }{ c }.
\end{align*}
Moreover, there exist $\beta$ and $\gamma$ such that $\mu_{l\beta} = l \beta \Fractional{\mu}(H)$ and $k_{l\gamma} = l \gamma \Fractional{k}(H)$ for all $l \in \N$. Therefore, $\Fractional{t}(G)$ is the limit of the time-per-budget ratio $t_b(G)/b$.

% Obviously, $t_b(G)$ and $\Fractional{t}(G)$ are related. Indeed, $\Fractional{t}(G)$ is the optimal value of the LP $P = M(H_{VC}(G))$, while $t_b(G)$ is the optimal value of the integer program $\IP{P_b}$.

\begin{proposition}
For any $G$, 
\[
	\Fractional{t}(G) = \lim_{b \to \infty} \frac{t_b(G)}{b} = \max_{b \to \infty} \frac{t_b(G)}{b}.
\]
Moreover, there exists $\beta \in \N$ such that $t_{l \beta}(G) = \Fractional{t}(G) \cdot l \beta$ for all $l \in \mathbb{N}$.
\end{proposition}

We now obtain more precise results about $t_b(G)$.

\begin{proposition}
For any $G$ and any $b$, we have
\[
	 \left\lfloor \frac{\chi(G)}{\chi(G) - 1} \cdot b \right\rfloor
	 \le t_b(G) \le
	 \left\lfloor \frac{\omega(G)}{\omega(G) - 1} \cdot b \right\rfloor
\]
\end{proposition}

\begin{proof}
If there is a homomorphism from $G'$ to $G$, which we denote as $G' \to G$, then $t_b(G) \le t_b(G')$. Since $K_{\omega(G)} \to G \to K_{\chi(G)}$, we obtain $t_b( K_{\chi(G)} ) \le t_b(G) \le t_b( K_{\omega(G)} )$. It is then easy to verify that $t_b( K_n ) = \left\lfloor \frac{n}{n - 1} \cdot b \right\rfloor$ for all $n \ge 1$. Hence the result.
\end{proof}

Since the chromatic number of a perfect graph can be computed in polynomial time \cite{GLS81}, we obtain the following

\begin{corollary}
If $G$ is a perfect graph, then for any $b$, $t_b(G) = \left\lfloor \frac{ \chi(G) }{ \chi(G) - 1 } \cdot b \right\rfloor$ can be computed in polynomial time.
\end{corollary}

The highest time $t_b(G)$ is only achieved for bipartite graphs, as seen below.

\begin{proposition} \label{proposition:bipartite}
The following are equivalent.
\begin{enumerate}[label=(\alph*)]
	\item \label{item:a} 
	$\Fractional{t}(G) = 2$.
	
	\item \label{item:b}
	$t_b(G) = 2b$ for some $b \ge 1$.
	
	\item \label{item:c}
	$t_b(G) = 2b$ for all $b \ge 1$.
	
	\item \label{item:d}
	$G$ is bipartite.
\end{enumerate}
\end{proposition}

\begin{proof}
We have \ref{item:d} $\implies$ \ref{item:c} $\implies$ \ref{item:b} $\implies$ \ref{item:a}. Conversely, $\Fractional{t}(G) = 2$ if and only if $\Fractional{\chi}(G) = 2$, which in turn is equivalent to $G$ being bipartite.
\end{proof}

% For any $c \ge 1$, the \Define{lexicographical product} of $G$ and the complete graph $K_c$ is 
% denoted as $G \cdot K_c$. Its vertex set is $V \times [c]$ and two distinct vertices $(u,a)$ and $(v,b)$ form an edge of $G \cdot K_c$ if and only if either $uv \in E(G)$ or $u=v$. The $c$-\Define{fold chromatic number} of $G$ is then $\chi_c(G) = \chi(G \cdot K_c)$. It is then well known that the \Define{fractional chromatic number} of $G$ then satisfies
% 	\[
% 		\Fractional{\chi}(G) = \lim_{c \to \infty} \frac{\chi_c(G)}{c} = \inf_{c \ge 1} \frac{\chi_c(G)}{c}.
% 	\]

We obtain a final result on the computational complexity of decision problems related to $t_b(G)$.

\begin{theorem} \label{theorem:mub_kc}
For any $b, c \ge 1$ and any hypergraph $H$,
\[
    \mu_b(H) \ge b + c \iff k_c( \HComplement{H} ) \le b + c.
\]
\end{theorem}

\begin{proof}
It is easy to verify that each statement is equivalent to the next, in the following sequence:
\begin{itemize}
	\item $\mu_b(H) \ge b+c$.
	
	\item There exist $b+c$ edges of $H$, say $e_1, \dots, e_{b+c}$, such that for any $v \in V$, $|\{i : v \in e_i\}| \le b$.

	\item There exist $b+c$ edges of $\HComplement{H}$, say $f_1, \dots, f_{b+c}$, such that for any $v \in V$, $|\{i : v \in f_i\}| \ge c$.

	\item $k_c(H) \le b + c$.
\end{itemize} 
\end{proof}

% For $H = H_{VC}(G)$, we have $\mu_b(H) = t_b(G)$ and $k_c(\HComplement{H}) = \chi_c(G)$, the smallest number of colours in a $c$-multicolouring of $G$. It follows from Theorem \ref{theorem:mub_kc} that for any $b \ge 1$ and $c \ge 1$, $t_b(G) \ge b + c$ if and only if $\chi_c(G) \le b + c$. Therefore, since for $c=1$, $\chi_c(G) = \chi(G)$ the chromatic number of $G$, we obtain the following corollary.

% \begin{corollary}
% For any $b \ge 2$, it is NP-complete to decide whether $t_b(G) \ge b + 1$.
% \end{corollary}

A $c$-multicolouring of a graph $G$ is a colouring of its vertices, such that each vertex is assigned a set of $c$ distinct colours, and where the sets of colours of any two adjacent vertices are disjoint \cite{BKPSW19}.  For $H = H_{VC}(G)$, we have $\mu_b(H) = t_b(G)$ and $k_c(\HComplement{H}) = \chi_c(G)$, the smallest number of colours in a $c$-multicolouring of $G$. It follows from Theorem \ref{theorem:mub_kc} that for any $b \ge 1$ and $c \ge 1$, $t_b(G) \ge b + c$ if and only if $\chi_c(G) \le b + c$. For $c=b$, as proved in Proposition \ref{proposition:bipartite}, deciding whether $t_b(G) = 2b$ can be done in polynomial time. On the other hand, since for any $c$ and any $a > 2c$, deciding whether a graph $G$ satisfies $\chi_c(G) \le a$ is NP-complete (see \cite[Section 3.9]{SU97}), we obtain the following corollary.

\begin{corollary}
For any $b \ge 2$ and any $1 \le c \le b - 1$, it is NP-complete to decide whether $t_b(G) \ge b + c$.
\end{corollary}

\end{document}